\documentclass{article}
\usepackage{graphicx} 
\usepackage{amsmath,amssymb,amsthm}
\usepackage{comment}
\usepackage{xcolor}
\usepackage{hyperref}

\newtheorem{thm}{Theorem}
\newtheorem{lem}[thm]{Lemma}
\newtheorem{prop}[thm]{Proposition}
\newtheorem{cor}[thm]{Corollary}
\theoremstyle{definition}
\newtheorem{dfn}[thm]{Definition}
\newtheorem{rmk}[thm]{Remark}
\newtheorem{eg}[thm]{Example}
 \numberwithin{thm}{section}
 \numberwithin{equation}{section}
 
\DeclareMathOperator\vol{vol} 
 
\DeclareMathOperator\Hess{Hess}

\def\R{\mathbb{R}}
\def\N{\mathbb{N}}
\def\d{\mathrm{d}}
\def\Z{\mathbb{Z}}

\newcommand{\ii}{\mathrm i}
\newcommand{\dd}{\mathrm d}
\newcommand{\Lie}{\mathcal L}
 
\newcommand{\ip}[2]{\langle #1,#2\rangle}
\newcommand{\norm}[1]{\lVert #1\rVert}

\title{Eigenforms and special holonomy}
\author{Bobby S.~Acharya \and Luc\'ia M.~Cabrera \and Simone Corbo \and Jason D.~Lotay}
\date{
}

\begin{document}

\maketitle

\begin{abstract}
We prove existence and non-existence results for $L^2$ eigenforms for the Hodge Laplacian on complete non-compact Ricci-flat manifolds with special holonomy, with a particular focus on gravitational instantons. We briefly describe the physical interpretation of these results in superstring and M-theory.   
\end{abstract}

\tableofcontents

\section{Introduction}\label{sec:intro}

The study of eigenvalues of the Hodge Laplacian acting on differential forms, and the associated eigenforms, is a classical topic in spectral geometry.  Whilst the spectrum of the Laplacian acting on functions has been extensively studied, the case of forms has received relatively little attention.  

Motivated both by Riemannian geometry and theoretical physics, in this article we are interested in $L^2$ eigenforms on complete non-compact Riemannian manifolds that have special holonomy.  There are many examples of such manifolds and they are of fundamental interest in the study of Einstein (and especially Ricci-flat) metrics, and in string theory and M-theory.    

Beyond the classical nature of the problem of understanding the $L^2$ spectrum from a mathematical perspective, $L^2$ eigenforms for the Hodge Laplacian with non-zero eigenvalue on certain special holonomy manifolds are interpreted physically as ``normalisable massive modes'', i.e.~they correspond to massive Kaluza--Klein particles/states.  Hence, understanding whether such eigenforms exist is of importance for understanding the physics of string theory or M-theory compactifications.

\subsection{Gravitational instantons}

Our main results concern \emph{gravitational instantons}.  

\begin{dfn}
    A \emph{gravitational instanton} is a complete, oriented, connected, non-compact Riemannian 4-manifold $(M^4,g)$ which is
    \begin{itemize}
        \item hyperk\"ahler, i.e.~the holonomy group of $g$ is contained in $\mathrm{SU}(2)$, and 
        \item has Riemann curvature in $L^2$.
    \end{itemize}
The first condition implies that the manifold is Ricci flat. The second condition implies that the Riemann curvature must decay at least quadratically at infinity with respect to distance from a fixed point.  
\end{dfn}

These 4-manifolds have received much attention in both the mathematics and theoretical physics communities (see e.g.~\cite{AcharyaJoyce,Barbosa,GibbonsHawking,Kronheimer,Minerbe2011}) and were recently classified \cite{SunZhang}.

\begin{rmk}\label{rmk:ALE-H}
The gravitational instantons for which the curvature decays faster than quadratically were first classified in \cite{Chens2,Chens1,Chens3}, building on work in  
\cite{Minerbe2010}.  
They split into 4 types: ALE, ALF, ALG and ALH. Each type corresponds to a volume growth of integer order decreasing from 4 to 1: more details are provided e.g.~in \cite{SunZhang}. 
\end{rmk}

\begin{rmk}\label{rmk:ALE}
    The ALE gravitational instantons were classified much earlier by Kronheimer \cite{Kronheimer} using a combination of hyperk\"ahler quotient and twistor methods and they are in one-to-one correspondence with the $ADE$ Lie algebras. In fact, an ALE hyperk\"ahler 4-manifold, $(M^4_{\Gamma_{ADE}}, g)$ is asymptotic to the flat 4-orbifold $\mathbb{R}^4/\Gamma_{ADE}$, where $\Gamma_{ADE}$ is a finite subgroup of $\mathrm{SU}(2)$ acting on $\mathbb{R}^4 \cong \mathbb{C}^2$ in the standard representation.
\end{rmk}

Some important examples of gravitational instantons include the following.
\begin{eg} The Euclidean metric on $\R^4$ and the Eguchi--Hanson metric on $T^*\mathcal{S}^2$ give ALE gravitational instantons, where the latter is asymptotic to $\R^4/\mathbb{Z}_2$.
\end{eg}

\begin{eg} The Taub--NUT metric on $\R^4$, the Atiyah--Hitchin metric on $\mathcal{S}^4\setminus\mathbb{RP}^2$ coming from viewing it as a monopole moduli space, and the double cover $\mathbb{CP}^2\setminus\mathbb{RP}^2$ of Atiyah--Hitchin give ALF gravitational instantons.
\end{eg}

\begin{eg} The Gibbons--Hawking ansatz provides all ALE and ALF gravitational instantons of $A_n$ type (also known as cyclic type), which includes all of the examples above except Atiyah--Hitchin and its double cover (which are $D_n$ type). These  ALE/ALF gravitational instantons are sometimes called multi-Eguchi--Hanson or multi-Taub--NUT respectively.
\end{eg}

\begin{rmk}\label{rmk:AL*}
If the curvature of a gravitational instanton does not decay faster than quadratically, then two more types of asymptotic behaviour  appear: ALG* and ALH*. These new types are, in some sense, intermediate behaviours between ALF and ALG, and ALG and ALH respectively, with volume growths of 2 (ALG*) and $4/3$ (ALH*). See e.g.~\cite[Section 6.4]{SunZhang} for detailed definitions. 
\end{rmk}

\subsection{Results}

To state our results it will be useful to introduce notation for the  $\lambda$-eigenspace for the Hodge Laplacian acting on $L^2$ $k$-forms on a complete Riemannian manifold $(M,g)$:
\begin{equation}
    \mathcal{E}^k(\lambda)=\{\alpha\in L^2\Omega^k(M,g)\,\colon \Delta\alpha=\lambda\alpha\}.
\end{equation}

Our first key result is the following.

\begin{thm}\label{thm:grav.inst}
    Let $(M^4,g)$ be a gravitational instanton and let $\lambda\in\R$, $\lambda\neq 0$.  
    \begin{itemize}
        \item[(a)] If $\mathcal{E}^0(\lambda)\neq 0$ then $\mathcal{E}^k(\lambda)\neq 0$ for all $k\in\{0,1,2,3,4\}$.
        \item[(b)] If $\mathcal{E}^0(\lambda)=0$ then $\mathcal{E}^k(\lambda)=0$ for all $k\in\{0,1,2,3,4\}$.
    \end{itemize}
\end{thm}

Theorem \ref{thm:grav.inst}(a) is not surprising and quite elementary. It extends naturally to other situations where the Riemannian metric has reduced holonomy as follows.

\begin{prop}\label{prop:function.forms.intro}
Let $(M^n,g)$ be a complete Riemannian manifold such that whenever $f\in\mathcal{E}^0(\lambda)$ for $\lambda\neq 0$ we have $\d f\in L^2$.
\begin{itemize}
\item[(a)] Suppose the holonomy of $g$ is contained in $\mathrm{U}(n/2)$ if $n$ is even (i.e.~$M$ is K\"ahler) or in $\mathrm{G}_2$ if $n=7$ (i.e.~$M$ is a $\mathrm{G}_2$ manifold).  If $\mathcal{E}^0(\lambda)\neq 0$ for $\lambda\neq 0$ then $\mathcal{E}^k(\lambda)\neq 0$ for all $k\in\{0,\ldots,n\}$.  
\item[(b)] Suppose $n=8$ and the holonomy of $g$ is contained in $\mathrm{Spin}(7)$ (i.e.~$M$ is a $\mathrm{Spin}(7)$ manifold).  If $\mathcal{E}^0(\lambda)\neq 0$ for $\lambda\neq 0$ then $\mathcal{E}^k(\lambda)\neq 0$ for all $k\in\{0,1,3,4,5,7,8\}$.
\end{itemize}
\end{prop}

\noindent The condition for $(M,g)$ to have $\d f\in L^2$ whenever $f\in\mathcal{E}^0(\lambda)$ for $\lambda\neq 0$ holds in many examples, as we demonstrate in Lemma \ref{lem:lambda.non.neg} and Lemma \ref{lem:integrability.grav.inst}.

\begin{rmk}
\begin{itemize}\item[]
\item Theorem \ref{thm:grav.inst}(b) contains more meaningful content and  physics predicts that it could fail for other special holonomy manifolds, namely Calabi--Yau 3 and 4-folds, $\mathrm{G}_2$ manifolds and $\mathrm{Spin}(7)$ manifolds.  In these settings we prove partial versions of Theorem \ref{thm:grav.inst}(b): see Lemmas \ref{lem:CY3.vanish}, \ref{lem:G2.vanish} and \ref{lem:Spin7.vanish}.     
\item Note that Theorem \ref{thm:grav.inst}(b) does not hold for $\lambda=0$, i.e.~for harmonic forms.  For example, any non-flat ALE gravitational instanton has $L^2$ harmonic 2-forms, but has no $L^2$ harmonic functions by the maximum principle.
\end{itemize}
\end{rmk}

Theorem \ref{thm:grav.inst} shows that to understand eigenforms on gravitational instantons it suffices to consider eigenfunctions. Proposition \ref{prop:function.forms.intro} and work in Section \ref{sec:special.hol} show that similar, though so far weaker, results  hold for Calabi--Yau 3-folds, $\mathrm{G}_2$ manifolds and Spin(7) manifolds.  

A particularly important class of non-compact complete Riemannian manifolds with special holonomy are those which are \emph{asymptotically conical} (AC): see Definition \ref{dfn:AC} for a formal definition.  There are many examples of such manifolds, particularly hyperk\"ahler 4-manifolds and Calabi--Yau manifolds, and they are well-studied: see e.g.~\cite{BryantSalamon, CandelasDeLaOssa, ConlonHein, FHN2, KL, Kronheimer, Lehmann, Lehmann2, Stenzel}. 
 Key mathematical reasons for the importance of AC special holonomy manifolds are that they have maximal volume growth (allowed by the Ricci-flat condition) and they give local models for how compact Einstein manifolds can degenerate in families to develop point singularities.  Physically, the conical singularities which appear are important in {\it geometric engineering of quantum field theories} 
\cite{Acharya, AtiyahWitten2003, IntriligatorMorrisonSeiberg1997, KatzKlemmVafa1997}. 
The results we prove here allow for rigorous statements in this physical framework.

\begin{prop}\label{prop:AC.E0}
 Let $(M,g)$ be an asymptotically conical Riemannian  manifold.  
Then we have $\mathcal{E}^0(\lambda)=0$ for all $\lambda\in\R$.   
\end{prop}

\begin{rmk}
As is well-known, and as we shall see explicitly in Corollary \ref{cor:ALF}, Proposition \ref{prop:AC.E0} can fail for complete non-compact Riemannian $n$-manifolds which have volume growth less than that of $\R^n$.
\end{rmk}

Putting Theorem \ref{thm:grav.inst} and Proposition \ref{prop:AC.E0} together yields the following.

\begin{cor}\label{cor:ALE}
    Let $(M^4,g)$ be an ALE gravitational instanton.  Then $\mathcal{E}^k(\lambda)=0$ for all $\lambda\neq 0$ and for all $k$.
\end{cor}

ALF gravitational instantons have received attention in mathematics and theoretical physics through connections to monopole moduli spaces  \cite{AtiyahHitchin} and gluing constructions for Ricci-flat metrics \cite{Foscolo,SchroersSinger}.  In particular, the spectrum of the Laplacian on ALF gravitational instantons is of special physical interest \cite{Schroers1,GibbonsManton,Schroers2}. We therefore prove the following result in the ALF setting, which extends work in \cite{Schroers2}.

\begin{thm}\label{thm:An.ALF}
Let $(M^4,g)$ be an $A_n$ type ALF (or multi-Taub--NUT) gravitational instanton.  Then we have $\mathcal{E}^0(\lambda)=0$ for all $\lambda\in\R$.
\end{thm}

Combining Theorem \ref{thm:An.ALF} with results from \cite{Schroers3,GibbonsManton} on eigenfunctions on the Atiyah--Hitchin manifold yields another corollary to Theorem \ref{thm:grav.inst}.

\begin{cor}\label{cor:ALF}
\begin{itemize}\item[]
    \item[(a)] On $A_n$ ALF/multi-Taub--NUT gravitational instantons we have $\mathcal{E}^k(\lambda)=0$ for all $\lambda\neq 0$ and for all $k$.
    \item[(b)] On Atiyah--Hitchin and its double cover there exist infinitely many $\lambda\in \R$, $\lambda\neq 0$, such that $\mathcal{E}^k(\lambda)\neq 0$ for all $k\in\{0,1,2,3,4\}$.
\end{itemize}
\end{cor}

\begin{rmk}
    It would be interesting to investigate whether or not Corollary \ref{cor:ALF}(b) extends to other $D_n$ type ALF gravitational instantons.
\end{rmk}

We now note that, by using methods from the scattering calculus  \cite{Melrose}, we can obtain non-existence of eigenforms in the general asymptotically conical setting using \cite{CNTV} (see also \cite[p.~78]{Melrose} for a relevant statement but without proof).  In particular, this result applies to AC special holonomy manifolds.

\begin{prop}\label{prop:AC.Ek}
    Let $(M,g)$ be an asymptotically conical Riemannian manifold.  Then $\mathcal{E}^k(\lambda)=0$ for all $\lambda\neq 0$ and all $k$.
\end{prop}

\begin{rmk}
  Corollary \ref{cor:ALE} follows directly from Proposition \ref{prop:AC.Ek} without requiring Theorem \ref{thm:grav.inst} (since ALE manifolds are AC), but we feel there is still value in seeing how the hyperk\"ahler condition reduces the eigenform problem to the eigenfunction setting.      
\end{rmk}

\begin{rmk}
There are important examples of special holonomy manifolds with maximal volume growth which are \emph{not} asymptotically conical, especially in the Calabi--Yau case \cite{ConlonRochon,Li,Sz}.  It would be interesting to know whether or not Proposition \ref{prop:AC.Ek} extends to these examples, since the methods of \cite{Melrose} do not apply.
\end{rmk}

\begin{rmk}
    Asymptotically cylindrical Riemannian manifolds with special holonomy have also received significant attention, e.g.~\cite{HHN,Kovalev,KovNor}.  In contrast to Propositions \ref{prop:AC.E0} and \ref{prop:AC.Ek}, one expects to have $L^2$ eigenfunctions with non-zero eigenvalue and thus, by Proposition \ref{prop:function.forms.intro}, $L^2$ eigenforms in many degrees on such manifolds.
\end{rmk}

\subsection{Summary}

We now briefly summarise the contents of the article.

In Section \ref{sec:dfns} we provide the basic definitions and notation that shall be used throughout.  In particular, we define various important classes of complete non-compact Riemannian manifolds and give  examples with special holonomy.

In Section \ref{sec:fundamental} we give some fundamental regularity and integrability results for $L^2$ eigenforms, which in particular show that their eigenvalues must be non-negative.  We interpret this latter result in terms of stability of special holonomy metrics.

In Section \ref{sec:functions} we focus on $L^2$ eigenfunctions.  We first prove the existence result Proposition \ref{prop:function.forms.intro}, showing how to obtain $L^2$ eigenforms from $L^2$ eigenfunctions in the presence of reduced holonomy using parallel forms.  We then show in the asymptotically conical setting that the non-existence result Proposition \ref{prop:AC.E0},  and its extension (Proposition \ref{prop:AC.Ek}) to eigenforms, follow quickly from known results in the literature.

In Section \ref{sec:grav.inst} we prove Theorem \ref{thm:grav.inst}(b), i.e.~that on gravitational instantons $(M^4,g)$ the non-existence of $L^2$ eigenfunctions implies the non-existence of $L^2$ eigenforms of any degree.  This is achieved through elementary arguments, such as the splitting of 2-forms on $(M,g)$ into self-dual and anti-self-dual 2-forms, and the triviality of the bundle of self-dual 2-forms on $(M,g)$.

Section \ref{sec:funct.ALF} is the most technical part of the article, where we prove Theorem \ref{thm:An.ALF}: there are no $L^2$ eigenfunctions on an ALF gravitational instanton $(M^4,g)$ of $A_{n}$/cyclic type.  The key idea is to use the circle action present on $(M,g)$ to perform a Fourier decomposition, thus reducing the eigenvalue problem on $(M,g)$ to a family of eigenvalue problems on $X_0=\R^3\setminus\{x_0,\ldots,x_n\}$ (where $n$ is the same index as appears in $A_{n}$ type) for sections of appropriate complex line bundles over $X_0$.  The main tool is an extension of arguments in \cite{AHK} for absence of eigenvalues for magnetic Schr\"odinger operators on Euclidean space.

In Section \ref{sec:special.hol} we give some partial extensions of Theorem \ref{thm:grav.inst}(b) to the setting of complete non-compact Calabi--Yau 3-folds, $\mathrm{G}_2$ and $\mathrm{Spin}(7)$ manifolds; i.e.~vanishing theorems for $L^2$ eigenforms based on the vanishing of $L^2$ eigenfunctions.  The proofs rely on the decomposition of forms arising from special holonomy and representations of the holonomy group.  

Finally, in Section 8, we give some brief discussion of $L^2$ eigenforms for the Laplacian and special holonomy from the point of view of string theory/M-theory, as this was the original main motivation for the authors to investigate the problems in this article.  

\subsection*{Acknowledgements}
BSA would like to thank L.~Foscolo and S.~Sun for discussions. JDL thanks R.~Mazzeo and A.~Vasy for helpful discussions. We also acknowledge the Instituto Balseiro, Bariloche for hospitality, where this work was partially completed.

\section{Definitions}\label{sec:dfns}

In this section we provide the key definitions and common notation we shall use in the article.  Throughout we assume our manifolds are connected and oriented.

\subsection{Eigenforms}
 We are interested in studying
eigenforms for the (Hodge) Laplacian $\Delta$ on a complete Riemannian manifold $(M,g)$:
\begin{equation}\label{eq:Delta}
    \Delta=\d\d^*+\d^*\d,
\end{equation}
where $\d^*$ is the formal adjoint of the exterior derivative $\d$.  By eigenforms, we mean non-zero $k$-forms $\alpha$ solving
\begin{equation}\label{eq:eigenform}
    \Delta\alpha=\lambda\alpha
\end{equation}
for a constant $\lambda$.  

\begin{rmk}
Our convention in \eqref{eq:Delta} is often called the ``geometer's Laplacian'' and is the negative of the ``analyst's Laplacian''.  For concreteness, on Euclidean space $\mathbb{R}^n$ and acting on functions \eqref{eq:Delta} becomes $-\sum_{j=1}^n\frac{\partial^2}{\partial x_j^2}$ for Euclidean coordinates $(x_1,\ldots,x_n)$.  The reader should be aware of this sign convention issue.
\end{rmk}

Our first assumption is that $\alpha$ satisfying \eqref{eq:eigenform} has the following natural normalisability condition:
\begin{equation}\label{eq:normalizability.1}
    \alpha\in L^2\Omega^k(M,g),
\end{equation}
so
\begin{equation}
    \int_M|\alpha|^2_g\vol_g<\infty.
\end{equation}
Since we are assuming that \eqref{eq:eigenform} holds, \eqref{eq:normalizability.1} forces 
\begin{equation}\label{eq:normalizability.2}
    \Delta\alpha\in L^2\Omega^k(M,g).
\end{equation}

We now provide some notation for eigenforms on a complete Riemannian manifold $(M,g)$ for convenience, as we did in Section \ref{sec:intro}.

\begin{dfn}\label{dfn:eigenforms}
For $\lambda\in\R$ we let
\begin{equation}
    \mathcal{E}^k(\lambda)=\{\alpha\in L^2\Omega^k(M,g)\,\colon \Delta\alpha=\lambda\alpha\}.
\end{equation}
We also let
\begin{equation}
    \mathcal{H}^k=\{\alpha\in L^2\Omega^k(M,g)\,\colon \Delta\alpha=0\}=\mathcal{E}^k(0)
\end{equation}
be the $L^2$ harmonic $k$-forms.
\end{dfn}

\begin{rmk}
    When $M$ is compact, we have that $\mathcal{H}^k$ is isomorphic to $H^k$, the $k$th de Rham cohomology group of $M$, by Hodge theory.  The proof of the Hodge theorem involves minimising the $L^2$ norm of closed forms  in their cohomology class, so the $L^2$ condition is natural.  This also leads to the study of various notions of $L^2$ cohomology on non-compact manifolds, where $\mathcal{H}^k$ is one possible way to generalise the theory of harmonic forms to the non-compact setting.  It remains an important research area to understand what topological or geometric information $\mathcal{H}^k$ encodes in the non-compact setting.
    \end{rmk}

    \begin{rmk}
Again, when $M$ is compact, one way to study eigenvalues of the Laplacian on $k$-forms is to consider the Rayleigh quotients
\begin{equation}\label{eq:Rayleigh}
    \frac{\|\d\alpha\|_{L^2}^2+\|\d^*\alpha\|_{L^2}^2}{\|\alpha\|_{L^2}^2}
\end{equation}
for $\alpha\neq 0$.   Minimising \eqref{eq:Rayleigh} over $\alpha\neq 0$ orthogonal to $\mathcal{H}^k$ yields the first eigenvalue $\lambda_1$ of $\Delta$.  One can then find higher eigenvalues for $\Delta$ by applying a min-max procedure for \eqref{eq:Rayleigh}.  As in the harmonic case, the $L^2$ norm  plays a key role, which motivates studying $\mathcal{E}^k(\lambda)$ in the non-compact setting.  However, our understanding of these spaces is rather limited.  
    \end{rmk}

\subsection{Asymptotic conditions}

We are interested in complete non-compact Riemannian manifolds with certain asymptotic behaviours.  The simplest  are those which are \emph{asymptotically conical} in the following sense.

\begin{dfn}\label{dfn:AC} A complete non-compact Riemannian $n$-manifold $(M^n,g)$ is \emph{asymptotically conical} (AC)  if there exist
\begin{itemize}
    \item a compact set $K\subseteq M$, 
    \item a compact Riemannian $(n-1)$-manifold $(\Sigma^{n-1},g_{\Sigma})$, so that $C=(0,\infty)\times\Sigma$ with coordinate $r\in(0,\infty)$ has the \emph{cone} metric
   \begin{equation}\label{eq:cone.metric}
       g_C=dr^2+r^2g_{\Sigma},
   \end{equation}
     \item a positive constant $R>0$,
    \item a negative constant $\nu<0$ and
    \item a diffeomorphism $\Psi:(R,\infty)\times\Sigma \to M\setminus K$ such that 
  \begin{equation}\label{eq:AC}
    |\nabla_C^j(\Psi^*g-g_C)|_{g_C}=O(r^{\nu-j})\quad\text{as $r\to\infty$}
    \end{equation}
    for all $j\in\mathbb{N}$, where $\nabla_C$ is the Levi-Civita connection of $g_C$.
\end{itemize}
We call $(C,g_C)$ the \emph{asymptotic cone} of $(M,g)$ and $\nu<0$ the \emph{rate} of $(M,g)$.  
\end{dfn}

\begin{rmk}
For a complete AC manifold with Ricci non-negative, which is not isometrically a product with a Euclidean factor, we know that $\Sigma$ must be connected by the Cheeger--Gromoll splitting theorem.
\end{rmk}

\begin{eg}\label{eg:ALE}
    ALE gravitational instantons are AC where the asymptotic cone $C$ is flat, i.e.~a quotient of $\R^4$ by a finite group of isometries fixing the origin.
\end{eg}

\begin{eg} There are many important examples of AC Riemannian manifolds with special holonomy (see e.g.~\cite{BryantSalamon,ConlonHein,FHN2,Lehmann}). These include Calabi--Yau metrics on $T^*\mathcal{S}^n$,  holonomy $\mathrm{G}_2$ metrics  on $\Lambda^2_-T^*N$ for $N=\mathcal{S}^4$ or $\mathbb{CP}^2$, and holonomy Spin(7) metrics on the spinor bundle of $\mathcal{S}^4$ and the nontrivial rank $3$ bundle over $\mathcal{S}^5$.
\end{eg}

We can generalise the notion of asymptotically conical in several directions.  One important way is to allow the asymptotic geometry at infinity to split into a torus factor and a cone or, more generally, into a torus bundle over a cone (cf.~\cite{Cavalleri}).   

\begin{dfn}\label{dfn:ATC} A complete non-compact Riemannian $n$-manifold $(M^n,g)$ is \emph{asymptotically $T^m$-fibred conical} (A$T^m$C) if there exist
\begin{itemize}
    \item a compact set $K \subseteq M$,
    \item a compact Riemannian $(n-m-1)$-manifold $(\Sigma^{n-m-1},g_{\Sigma})$, so that $C=(0,\infty)\times\Sigma$ with coordinate $r\in(0,\infty)$ has the cone metric $g_C$ as in \eqref{eq:cone.metric},
    \item a positive constant $R>0$,
    \item a principal $T^m$-bundle $\pi:P\to (R,\infty)\times\Sigma$ with a radially invariant metric $g_{T^m}$ on the fibres so that each fibre is a flat torus,
    \item a radially invariant connection $A_{\infty}$ on $P$ leading to a metric
    \begin{equation}
        g_P=\pi^*g_C+g_{T},
    \end{equation}
    where $g_T$ is zero on the horizontal spaces of $A_{\infty}$ and equal to $g_{T^m}$ on the vertical spaces,
    \item a negative constant $\nu<0$ and
    \item a diffeomorphism $\Psi:P\to M\setminus K$ such that for all $j\in\N$ we have
    \begin{equation}
        |\nabla^j_C(\Psi^*g-g_P)|_{g_P}=O(r^{\nu-j})\quad\text{and}\quad|\nabla^j_T(\Psi^*g-g_P)|=O(r^{\nu})\quad\text{as $r\to\infty$,}
    \end{equation}
    where $\nabla_C$ and $\nabla_T$ denote the splitting of the Levi-Civita connection of $g_P$ into the horizontal and vertical directions on $P$ respectively.  We also impose appropriate decay conditions on mixed $\nabla_C$ and $\nabla_T$ derivatives of $\Psi^*g-g_P$: see \cite[Definition 3.13]{Cavalleri} for details.
\end{itemize}

\end{dfn}

\begin{rmk} If we take the convention that the ``zero torus'' $T^0$ is a point, then A$T^0$C is just AC.  The case of A$T^1$C is often called \emph{asymptotically local conical} (ALC).
\end{rmk}

\begin{eg}\label{eg:ALF-H} All
 ALF and some ALG and ALH gravitational instantons are A$T^m$C   where $m=1,2,3$ as appropriate and  the cone $C$ is flat.  In general ALG and ALH gravitational instantons are not strictly A$T^m$C as they are only asymptotic to torus fibre bundles, but after passing to a finite cover they become A$T^m$C, so can be treated using the same  methods.   However, ALG* and ALH* gravitational instantons do not fit into the A$T^m$C framework.
\end{eg}

\begin{eg}
There are many examples of A$T^1$C/ALC special holonomy manifolds, including $\mathrm{G}_2$ manifolds (e.g.~\cite{FHN1,FHN2}) and Spin(7) manifolds (e.g.~\cite{FoscoloSpin,Lehmann}).  There are also examples of many A$T^2$C Spin(7) manifolds given in \cite{Cavalleri}.
\end{eg}

\begin{rmk}
    There are other types of asymptotic behaviour one could consider and which appear in the study of special holonomy.  For example,  asymptotically cylindrical examples have been quite well studied (e.g.~\cite{HHN,Kovalev,KovNor}).  We do not study these other cases explicitly, but several of our results either directly apply or should extend to other asymptotic regimes.  We endeavour to indicate this where it is appropriate.
\end{rmk}

\section{Fundamentals}\label{sec:fundamental}

In this section we prove some fundamental results for $L^2$ eigenforms for the Laplacian and their corresponding eigenvalues.  Specifically, we prove a regularity result and an integrability result, which demonstrates that the eigenvalues are non-negative.  We also make some brief interpretations for this result for special holonomy manifolds in terms of stability.

\subsection{Regularity}
The conditions \eqref{eq:normalizability.1} and \eqref{eq:normalizability.2}  imply the following local regularity result.

\begin{lem}\label{lem:smooth}
    An $L^2$ eigenform for the Laplacian is smooth: i.e.~a solution $\alpha$ to \eqref{eq:eigenform} satisfying \eqref{eq:normalizability.1} must also satisfy $\alpha\in C^{\infty}$. 
\end{lem}

\begin{proof}
    Since $\Delta\alpha\in L^2$ by \eqref{eq:normalizability.2}, we have by elliptic regularity that $\alpha\in L^2_{2,\text{loc}}$. It then follows that $\Delta\alpha\in L^2_{2,\text{loc}}$ by \eqref{eq:eigenform}.  Iterating this argument means that $\alpha\in L^2_{k,\text{loc}}$ for all $k\in\mathbb{N}$.  Sobolev embedding now yields that $\alpha\in C^k_{\text{loc}}$ for all $k\in\mathbb{N}$, which gives the result.
\end{proof}

\subsection{Non-negativity of eigenvalues}

It is well-known that the eigenvalues  of the Laplacian \eqref{eq:Delta} on $k$-forms on a compact manifold are non-negative.  One reason is that eigenvalues are given via Rayleigh quotients as in \eqref{eq:Rayleigh}.  These considerations motivate us to  consider whether $\d\alpha,\d^*\alpha\in L^2$ if $\alpha\in\mathcal{E}^k(\lambda)$.   We show that this is indeed the case in the setting of AC and A$T^m$C manifolds, from which we deduce that the eigenvalues of the Laplacian on $k$-forms are non-negative.

\begin{lem}\label{lem:lambda.non.neg} Let $(M,g)$ be an AC or A$T^m$C Riemannian manifold. Let $\alpha\neq 0$ be a differential form satisfying \eqref{eq:eigenform} for $\lambda\in\R$ and \eqref{eq:normalizability.1}.  Then $\d\alpha,\d^*\alpha\in L^2$ and
\begin{equation}\label{eq:lambda.non.neg}
\lambda\|\alpha\|_{L^2}^2=\langle \alpha,\Delta\alpha\rangle_{L^2}=\|\d\alpha\|_{L^2}^2+\|\d^*\alpha\|_{L^2}^2\geq 0,
\end{equation}
so $\lambda\geq 0$.
\end{lem}

\begin{proof}
Note that from \eqref{eq:normalizability.1} and \eqref{eq:normalizability.2} we have that $\alpha,\Delta\alpha\in L^2$.  Suppose that we have proved that $\d\alpha,\d^*\alpha\in L^2$.  Then, the usual integration by parts formula is valid and so
\begin{equation}
\langle\alpha,\Delta\alpha\rangle_{L^2}=\|\d\alpha\|_{L^2}^2+\|\d^*\alpha\|_{L^2}^2,
\end{equation}
which immediately gives \eqref{eq:lambda.non.neg} as required.  Therefore, it suffices to prove that $\d\alpha,\d^*\alpha\in L^2$ to complete the proof.

To this end, we adapt the proof in \cite[pp.~340--341]{Karp}.   The idea is to perform the usual integration by parts as above but on compact sets which form an exhaustion of $M$, then use the fact that $\alpha,\Delta\alpha\in L^2$ to deduce that $\d\alpha,\d^*\alpha\in L^2$.

Recall the notation from Definitions \ref{dfn:AC} and \ref{dfn:ATC}.  In the AC case, for each $k\in\N$ let \begin{equation} 
U_k=K\cup \Psi\big((R,R+2k+1)\times\Sigma\big).
\end{equation}
In the A$T^m$C case, for each $k\in\N$ we let 
\begin{equation}
    P_k=\pi^{-1}\big((R,R+2k+1)\times\Sigma\big),
\end{equation}
i.e.~the total space of the bundle $P$ over a truncation of the cone $C$, and
\begin{equation}
    U_k=K\cup \Psi(P_k).
\end{equation}
In both cases, $U_k$ is an open set with compact closure in $M$ for every $k$, and $\bigcup_{k\in \N}U_k=M$.  

For each $k\in\N$ we then choose a smooth cut-off function $f_k:M\to [0,1]$ such that 
\begin{equation}
    f_k(x)=\left\{\begin{array}{cl} 1 & x\in U_k, \\
    0 & x\in  M\setminus\overline{U_{k+1}}
    \end{array}\right.\quad\text{and}\quad |\nabla f_k|_g\leq 1.
\end{equation}
Such functions exist by choosing each to be $1$ on $K$ and then using composition with $\Psi$ and radial cut-off functions on $C=(0,\infty)\times\Sigma$. 

As argued in the proof of Lemma \ref{lem:smooth}, we know that $\Delta\alpha\in L^2$ and $\alpha\in L^2_{2,\text{loc}}$.  
Therefore, we may compute:
\begin{align}
    \langle f_k\alpha,f_k\Delta\alpha\rangle_{L^2}&=\int_M f_k^2\d\d^*\alpha\wedge*\alpha +\int_Mf_k^2\alpha\wedge *\d^*\d\alpha\\
    &=\int_M\d(f_k^2\d^*\alpha\wedge*\alpha)+(-1)^{k+1}\int_M \d^*\alpha\wedge \d(f_k^2*\alpha)\label{eq:int.parts.1}\\
    &\qquad-\int_M \d(f_k^2\alpha\wedge *\d\alpha)+\int_M\d(f_k^2\alpha)\wedge *\d\alpha.\label{eq:int.parts.2}
\end{align}
The first terms in \eqref{eq:int.parts.1} and \eqref{eq:int.parts.2} vanish by Stokes' theorem since $f_k^2\d^*\alpha\wedge *\alpha$ and $f_k^2\alpha\wedge *\d\alpha$ are compactly supported.  For the second term in \eqref{eq:int.parts.1} we see that
\begin{align}
 (-1)^{k+1}\int_M \d^*\alpha\wedge \d(f_k^2*\alpha)&=\|f_k\d^*\alpha\|_{L^2}^2+2(-1)^{k+1}\int_Mf_k\d^*\alpha\wedge\d  f_k\wedge *\alpha.
\end{align}
 Similarly, for the second term in \eqref{eq:int.parts.2} we have
 \begin{align}
     \int_M\d(f_k^2\alpha)\wedge *\d\alpha=\|f_k\d\alpha\|_{L^2}^2+2\int_Mf_k\d f_k\wedge\alpha\wedge *\d\alpha.
 \end{align}
Since \eqref{eq:eigenform} is satisfied, we have that
\begin{multline}
    \lambda\|f_k\alpha\|_{L^2}^2-\|f_k\d\alpha\|_{L^2}^2-\|f_k\d^*\alpha\|_{L^2}^2\\
    =-2\int_M f_k\d f_k\wedge \d^*\alpha\wedge *\alpha +2\int_M f_k \d f_k\wedge \alpha\wedge *\d\alpha.
\end{multline}
Using $2ab\leq \frac{1}{2}a^2+2b^2$, we see that
\begin{align}
    2\left|\int_M  f_k\d f_k\wedge \d^*\alpha\wedge *\alpha\right|&\leq \frac{1}{2}\|f_k\d^*\alpha\|_{L^2}^2+2\|\d f_k\wedge *\alpha\|_{L^2}^2\\
    & \leq \frac{1}{2}\|f_k\d^*\alpha\|_{L^2}^2+2\|\alpha\|_{L^2}^2,
\end{align}
since $|\nabla f_k|_g\leq 1$.  An almost identical argument allows us to show that
\begin{equation}
2\left|\int_M f_k \d f_k\wedge \alpha\wedge *\d\alpha\right|\leq \frac{1}{2}\|f_k\d\alpha\|^2_{L^2}+2\|\alpha\|_{L^2}^2.
\end{equation}
Combining these observations, we see that
\begin{equation}
    \|f_k\d\alpha\|_{L^2}^2+\|f_k\d^*\alpha\|_{L^2}^2\leq c\|\alpha\|_{L^2}^2
\end{equation}
for some constant $c>0$ (depending on $\lambda$).  Since the right-hand side is independent of $k$, we can send $k\to \infty$ and deduce that $\d\alpha,\d^*\alpha\in L^2$.  
The result then follows.
\end{proof}

We also note the following in the case of gravitational instantons.

\begin{lem}\label{lem:integrability.grav.inst}  Let $(M^4,g)$ be a gravitational instanton.  
   Let $\alpha\neq 0$ be a differential form satisfying \eqref{eq:eigenform} for $\lambda\in\R$ and \eqref{eq:normalizability.1}.  Then $\d\alpha,\d^*\alpha\in L^2$ and \eqref{eq:lambda.non.neg} holds, so $\lambda\geq 0$.
\end{lem}

\begin{proof} For the case of ALE, ALF, ALG and ALH gravitational instantons the result follows from Lemma \ref{lem:lambda.non.neg} by the observations in Examples \ref{eg:ALE} and \ref{eg:ALF-H}.

By Remark \ref{rmk:AL*}, the only remaining cases are if $(M^4,g)$ is an ALG* or ALH* gravitational instanton. Then, outside of a compact set, $M$ can be identified with a model space on which one can define a radial distance function, just as in the AC and A$T^m$C settings.  Hence, one can adapt  the construction of the exhaustion by sets $U_k\subseteq M$ and cut-off functions $f_k:M\to [0,1]$ for $k\in\N$ in the proof of Lemma \ref{lem:lambda.non.neg} so that it carries through in these cases.
\end{proof}

\begin{rmk}
As indicated in the proof of Lemma \ref{lem:integrability.grav.inst}, Lemma \ref{lem:lambda.non.neg} extends to other types of complete non-compact Riemannian manifolds, such as asymptotically cylindrical ones, which arise in special holonomy. 
\end{rmk}

\subsection{Special holonomy and stability}

On Ricci-flat special holonomy manifolds, the linearisation of Ricci curvature under variations of the metric (after gauge-fixing for the action of diffeomorphisms and conformal transformations) is equivalent to the Laplacian on a certain space of forms.  Lemma \ref{lem:lambda.non.neg} shows that there are no ``negative modes'' for varying the Ricci curvature on AC or A$T^m$C Ricci-flat special holonomy manifolds. This means in particular that, at the linear level, there are no strictly destabilising directions for such Ricci-flat metrics under the Ricci flow, or for the Einstein--Hilbert functional.  

However, as shown in \cite{Chi}, there are examples of AC Ricci-flat special holonomy manifolds that can be deformed to AC Ricci-flat manifolds which are now no longer special holonomy.  In other words, the ``zero modes'' that define Ricci-flat deformations of an AC Ricci-flat special holonomy manifold do not necessarily correspond to special holonomy deformations: this is in direct contrast to the compact   and asymptotically cylindrical cases where such phenomena cannot occur \cite{Kovalev,Nordstrom,McKWang}.

\section{Functions}\label{sec:functions}

In this section we primarily study $L^2$ eigenfunctions for the Laplacian with non-zero eigenvalues.  We prove that the existence of such eigenfunctions guarantees the existence of $L^2$ eigenforms for the Laplacian of other degrees in the presence of appropriate reduced holonomy.  We also prove in the asymptotically conical setting that there are no such $L^2$ eigenfunctions and the extension of this non-existence result to $L^2$ eigenforms.

\subsection{Eigenfunctions and parallel forms}

We first show that we can obtain eigenforms from eigenfunctions in the presence of suitable parallel forms on the manifold, which exactly puts us in the setting of reduced holonomy.  This is a restatement of Proposition \ref{prop:function.forms.intro}.  

\begin{prop}\label{prop:functions.forms}
Let $(M^n,g)$ be a complete Riemannian manifold such that whenever $f\in \mathcal{E}^0(\lambda)$ for $\lambda\neq 0$ we have $\d f\in L^2$. 
\begin{itemize}
\item[(a)] Suppose the holonomy of $g$ is contained in $\mathrm{U}(n/2)$ if $n$ is even  
or in $\mathrm{G}_2$ if $n=7$.  
 If $\mathcal{E}^0(\lambda)\neq 0$ for $\lambda\neq 0$ then $\mathcal{E}^k(\lambda)\neq 0$ for all $k\in\{0,\ldots,n\}$.  
\item[(b)] Suppose $n=8$ and the holonomy of $g$ is contained in $\mathrm{Spin}(7)$.  
 If $\mathcal{E}^0(\lambda)\neq 0$ for $\lambda\neq 0$ then $\mathcal{E}^k(\lambda)\neq 0$ for all $k\in\{0,1,3,4,5,7,8\}$.
\end{itemize}
\end{prop}

\begin{proof}
It is an elementary fact that, on functions,
\begin{equation}\label{eq:Delta.Rough}
    \Delta=\nabla^*\nabla.
\end{equation}
In the case when $g$ has holonomy contained in $\mathrm{G}_2$ or $\mathrm{Spin}(7)$, then $g$ is Ricci flat so \eqref{eq:Delta.Rough} also holds on 1-forms.  These observations are useful in the following arguments.

Suppose first that we are in the K\"ahler case.  Then the K\"ahler form $\omega$ is parallel on $M$ and thus has constant norm. Moreover, by the K\"ahler identities, wedge product with $\omega$ (the Lefschetz operator) commutes with the Hodge Laplacian $\Delta$.  Hence, if $f\in\mathcal{E}^0(\lambda)$ then
\begin{equation}
    \Delta(f\omega^k)=(\Delta f)\omega^k=\lambda f\omega^k
\end{equation}
 and $f\omega^k\in L^2$ as $\omega^k$ has constant norm.
Thus, $\mathcal{E}^{2k}(\lambda)\neq 0$ for all $2k\leq n$ if $\mathcal{E}^0(\lambda)\neq 0$.  To get odd degree eigenforms we note that if $f\in\mathcal{E}^0(\lambda)$, $f\neq 0$ then $\d f\in\mathcal{E}^1(\lambda)$ with $\d f\neq 0$ since $f,\d f\in L^2$. Then
\begin{equation}
    \Delta(\d f\wedge\omega^k)=\Delta(\d f)\wedge\omega^k=\lambda\d f\wedge\omega^k.
\end{equation}
We deduce that $\mathcal{E}^{2k+1}(\lambda)\neq 0$ for all $2k+1\leq n$ if $\mathcal{E}^0(\lambda)\neq 0$.

Suppose now that we are in the $\mathrm{G}_2$ setting, which means that we have a parallel 3-form $\varphi$ and 4-form $*\varphi$, both with constant norm.  If $f\in\mathcal{E}^0(\lambda)$ with $f\neq 0$ then $f\varphi$ and $f*\varphi$ give non-zero elements in $\mathcal{E}^3(\lambda)$ and $\mathcal{E}^4(\lambda)$. As above, we also have $\d f\in\mathcal{E}^1(\lambda)\setminus\{0\}$.  Using \eqref{eq:Delta.Rough} and the fact that $*\varphi$ is parallel we see that $\d f\wedge *\varphi\in\mathcal{E}^5(\lambda)\setminus\{0\}$.  To get non-zero eigenforms in the remaining degrees 2 and 6 we may just use the fact that the Hodge star and $\Delta$ commute.  This completes the proof of (a).

In the $\mathrm{Spin}(7)$ setting we have a parallel 4-form, so the same arguments as in the $\mathrm{G}_2$ case  immediately give (b).  
\end{proof}

\begin{eg}  Applying
Proposition \ref{prop:functions.forms}(a) in the case of holonomy $\mathrm{Sp}(1)\cong\mathrm{SU}(2)\subseteq\mathrm{U}(2)$  proves Theorem \ref{thm:grav.inst}(a).
\end{eg}

\begin{rmk}
Note that  in the $\mathrm{Spin}(7)$ case (Proposition \ref{prop:functions.forms}(b)) we do not guarantee that if $\mathcal{E}^0(\lambda)\neq 0$ then $\mathcal{E}^2(\lambda)\neq 0$.  However, we do know it is true in the special case that the $\mathrm{Spin}(7)$-manifold is a Calabi--Yau 4-fold, for example, by Proposition \ref{prop:functions.forms}(a).
\end{rmk}

\subsection{AC eigenfunctions}

We now turn to the case of asymptotically conical manifolds, where we will need to use a distinguished function, called a \emph{radius} function. 

\begin{lem}\label{lem:radius} Let $(M,g)$ be an AC Riemannian manifold with rate $\nu<0$ and recall the notation of Definition \ref{dfn:AC}. Let $d_g$ be the distance function determined by $g$ and fix a point $p\in K$. 
    There exists a smooth function $\rho:M\to [0,\infty)$ such that  on $M\setminus K$ we have
    \begin{itemize}
\item[(a)]    $cd_g(\cdot,p)\leq \rho\leq c^{-1}d_g(\cdot,p)$ for some $c>0$,
\item[(b)] $||\nabla\rho|_g-1|=O(\rho^{\nu})$ as $\rho\to\infty$,
\item[(c)] $|\Hess\rho^2 -2g|_g=O(\rho^{\nu})$ as $\rho\to\infty$.
    \end{itemize}
    We call $\rho$ a radius function on $(M,g)$.
\end{lem}

\begin{proof}
Recall the notation of Definition \ref{dfn:AC}. 
We define $\rho$ on $\Psi((R+1,\infty)\times\Sigma)$ by 
\begin{equation}\label{eq:radius}
    \rho(\Psi(r,\sigma))=r.
\end{equation}
We then smoothly extend $\rho$ over the remainder of $M$ to be globally non-negative and positive outside $K$.  The AC condition implies that $\rho$ is equivalent to $d_g(.,p)$ outside a compact set, which leads to (a).  Since $|\nabla_Cr|=1$ and $|\Psi^*g-g_C|=O(r^{\nu})$ by \eqref{eq:AC}, we deduce that (b) is satisfied.  Finally, since $\Hess r^2=2g_C$ we can deduce (c) by again using \eqref{eq:AC}.  
\end{proof}


A direct consequence of Lemma \ref{lem:radius} is the following result, which is a restatement of Proposition \ref{prop:AC.E0}.

\begin{prop}\label{thm:noL^2.eigenfunctions}
Let $(M,g)$ be an AC Riemannian manifold.  Then $\mathcal{E}^0(\lambda)=0$ for all $\lambda\in\R$. 
\end{prop}

\begin{proof}  The radius function in Lemma \ref{lem:radius} gives an exhaustion function (as defined in \cite{Donnelly}) with precisely the required properties \cite[Properties 4.1]{Donnelly} to apply \cite[Corollary 5.3]{Donnelly}, from which the result follows.
\end{proof}

We also observe the following, which is a restatement of Proposition \ref{prop:AC.Ek}.

\begin{prop}\label{prop:noL^2.eigenforms.flat}
    Let $(M,g)$ be an AC Riemannian manifold.  Then $\mathcal{E}^k(\lambda)=0$ for all $\lambda\neq 0$ and all $k$. 
\end{prop}

\begin{proof}
We have already shown in Lemma \ref{lem:lambda.non.neg} that $\mathcal{E}^k(\lambda)=0$ for $\lambda<0$ for all $k$.  So assume that $\alpha\in\mathcal{E}^k(\lambda)$ for some $\lambda>0$ and some $k$.

We claim that the conditions of being AC and the form of the Hodge Laplacian mean that we can apply \cite[Corollary 5.5]{CNTV}, which relies on the scattering calculus (cf.~\cite{Melrose}) and considers eigenvalue problems $H\alpha=\lambda\alpha$ for appropriate perturbations $H$ of the rough Laplacian.  (Note that they consider sections of vector bundles in \cite{CNTV}.) 

First,  AC metrics are scattering metrics in the notation there. Second, we have that on forms
    \begin{equation}
        \Delta=\nabla^*\nabla+\mathcal{R},
    \end{equation}
  where $\mathcal{R}$ is an algebraic operator in the curvature of $(M,g)$.  Since $(M,g)$ is AC, $\mathcal{R}$ decays at infinity.    Recall the notation of Definition \ref{dfn:AC} and consider $\Psi(U\times (R,\infty))$ for some chart $U$ on $\Sigma$.  
  
  Choosing an orthonormal basis for the $k$-forms on $\Psi(U\times (R,\infty))$  and writing any $k$-form as a combination of these basis forms with coefficients $f_j$, we can then compare the rough Laplacian on $k$-forms with the scalar rough Laplacian on the $f_j$.  As $(M,g)$ is AC, the difference between rough Laplacians is a differential operator of order at most $1$ on the $f_j$ whose coefficients decay at infinity, and this holds for any chart $U$ on $\Sigma$. 
  
  Combining these observations with the decay of $\mathcal{R}$ means that we can take $H=\Delta$ in the notation of \cite[Corollary 5.5]{CNTV}.    This result states that solutions to $H\alpha=\lambda\alpha$ vanish outside a compact set on $M$. 
  Since $\alpha$ satisfies a linear elliptic equation, we deduce from unique continuation that $\alpha=0$ on $M$. 
\end{proof}

\begin{rmk}
It is relatively straightforward to see, using methods from separation of variables and ordinary differential equations, that there can be no eigenforms for $\Delta$ with non-zero eigenvalue on a Riemannian cone which lie in $L^2$, even when restricting to the complement of a neighbourhood of the vertex: see e.g.~\cite{Cheeger}.  This motivates why Propositions \ref{thm:noL^2.eigenfunctions} and \ref{prop:noL^2.eigenforms.flat} are true.  
\end{rmk}

\section{Eigenforms on gravitational instantons}\label{sec:grav.inst}

In this section we specialise to the case where $(M^4,g)$ is a gravitational instanton.  Our goal is to prove Theorem \ref{thm:grav.inst}(b), i.e.~that the non-existence of $L^2$ eigenfunctions implies the non-existence of $L^2$ eigenforms in any degree.

Before we begin, we recall that on an oriented Riemannian 4-manifold the 2-forms split into self-dual $\Omega^2_+(M)$ and anti-self-dual 2-forms $\Omega^2_-(M)$.  Since the Hodge star commutes with the Laplacian, $\Delta$ preserves the self-dual/anti-self-dual 2-forms, so we have splittings
\begin{equation}\label{eq:E2+-}
    \mathcal{E}^2(\lambda)=\mathcal{E}^2_+(\lambda)\oplus\mathcal{E}^2_-(\lambda).
\end{equation}

\subsection{Self-dual 2-forms}

An important fact about a hyperk\"ahler 4-manifold $(M^4,g)$ is that, given its natural orientation, its bundle of self-dual 2-forms $\Lambda^2_+T^*M$ is trivial and has three parallel everywhere orthogonal sections (of the same norm) $\omega_1,\omega_2,\omega_3$ defining the hyperk\"ahler structure.  In particular, $\omega_1,\omega_2,\omega_3$ are harmonic and are K\"ahler forms.  We begin with the following.

\begin{lem}\label{lem:hk.ALE.sd}   
If $(M^4,g)$ is complete hyperk\"ahler then $\mathcal{E}^2_+(\lambda)\neq 0$ if and only if $\mathcal{E}^0(\lambda)\neq 0$.
\end{lem}

\begin{proof}
 Any $\alpha\in L^2\Omega^2_+(M)$ can be written uniquely as 
 \begin{equation}\label{eq:sd.decomp.1}
     \alpha=f_1\omega_1+f_2\omega_2+f_3\omega_3
 \end{equation}
 for functions $f_1,f_2,f_3\in L^2$ (since $\omega_1,\omega_2,\omega_3$ have constant norm).   
 Moreover, on $\Omega^2_+(M)$ we have that \eqref{eq:Delta.Rough} holds (the Hodge and rough Laplacians agree)  
 because the Levi-Civita connection on $\Lambda^2_+T^*M$ is flat and trivial.  Since $\omega_1,\omega_2,\omega_3$ are parallel, we deduce that
 \begin{equation}\label{eq:sd.decomp.2}
     \Delta\alpha= 
     \sum_{j=1}^3(\Delta f_j)\omega_j.
 \end{equation}
Combining \eqref{eq:sd.decomp.1} and \eqref{eq:sd.decomp.2} we see that  $\alpha\in\mathcal{E}^2_+(\lambda)$ if and only if $f_j\in\mathcal{E}^0(\lambda)$ for all $j$, which gives the result.
\end{proof}

\subsection{1-forms}

We now move on to eigen-1-forms but first note the following useful result, which shows when exact self-dual or anti-self-dual 2-forms are necessarily zero.

\begin{lem}\label{lem:hk.asd.exact}  Let $(M^4,g)$ be a gravitational instanton.  
If $\alpha\in L^2\Omega^1(M)$ and $\d\alpha\in L^2\Omega^2_{\pm}(M)$ then $\d\alpha=0$.
\end{lem}

\begin{proof}
The idea is to   use integration by parts or Stokes' theorem with boundary at infinity.  The key point is that the condition that $\alpha,\d\alpha\in L^2$ ensures the boundary term at infinity vanishes, since 
\begin{equation}\label{eq:hk.asd.1}
    \|\d\alpha\|_{L^2}^2=\int_M\d\alpha\wedge *\d\alpha=\pm\int_M\d\alpha\wedge\d\alpha=\pm\int_M\d(\alpha\wedge\d\alpha).
\end{equation}
Concretely, since \begin{equation}
\alpha\wedge\d\alpha\in L^1\end{equation} and 
\begin{equation}
\d(\alpha\wedge\d\alpha)=\d\alpha\wedge\d\alpha\in L^1
\end{equation}we may apply Stokes' theorem for complete manifolds as in \cite{Gaffney} to deduce that the integral on the right-hand side of \eqref{eq:hk.asd.1} is zero.  Hence
\begin{equation}
    \|\d\alpha\|_{L^2}^2=0,
\end{equation}
which gives the result.
%
\end{proof}

\begin{prop}\label{prop:hk.ALE.1forms} Let $(M^4,g)$ be a gravitational instanton.  If $\mathcal{E}^0(\lambda)=0$ for $\lambda\neq 0$ then $\mathcal{E}^1(\lambda)=0$.
\end{prop}

\begin{proof}  Let $\alpha\in \mathcal{E}^1(\lambda)$.  
 We first note that by Lemma \ref{lem:integrability.grav.inst} we have $\d^*\alpha,\d\alpha\in L^2$.  We then see that
 \begin{equation}\label{eq:Delta.d*}
      \Delta \d^*\alpha=\d^*\Delta\alpha=\lambda\d^*\alpha.
 \end{equation}
 Hence $\d^*\alpha\in \mathcal{E}^0(\lambda)=0$ by hypothesis.

 We also see that 
\begin{equation}
     \d_+\alpha=\tfrac{1}{2}(\d\alpha+*\d\alpha)\in L^2\Omega^2_+(M).
\end{equation}
 Again we have that
\begin{equation}
    \Delta \d_+\alpha=\d_+\Delta\alpha=\lambda\d_+\alpha
\end{equation}
and so, by Lemma \ref{lem:hk.ALE.sd}, we know that $\d_+\alpha=0$. We then observe that
\begin{equation}
    \d\alpha=\tfrac{1}{2}(\d\alpha-*\d\alpha)=\d_-\alpha\in L^2\Omega^2_-(M).
\end{equation}
We may deduce that $\d\alpha=0$ from Lemma \ref{lem:hk.asd.exact}.

Given that $\d\alpha=0=\d^*\alpha$, we have that $\Delta\alpha=0$, but as $\lambda\neq 0$ this forces $\alpha=0$ as claimed.
\end{proof}

\subsection{Anti-self-dual 2-forms}

We may now conclude our study of eigenforms on gravitational instantons by considering anti-self-dual eigen-2-forms.

\begin{lem}\label{lem:hk.ALE.asd} Let $(M^4,g)$ be a gravitational instanton.  If $\mathcal{E}^0(\lambda)=0$ for $\lambda\neq 0$ then $\mathcal{E}^2_-(\lambda)=0$. 
\end{lem}

\begin{proof}  Let $\alpha\in\mathcal{E}^2_-(\lambda)$.  
By Lemma \ref{lem:integrability.grav.inst} we know that $\d^*\alpha\in L^2\Omega^1(M)$ and by \eqref{eq:Delta.d*} we know that $\d^*\alpha$ is a $\lambda$-eigenform for $\Delta$.  By Proposition \ref{prop:hk.ALE.1forms} we have that $\d^*\alpha=0$.  Moreover, since $\alpha\in\Omega^2_-(M)$, we know that
\begin{equation}
    \d^*\alpha=-*\d*\alpha=*\d\alpha.
\end{equation}
Hence, $\d\alpha=0$ as well.  Thus $\Delta\alpha=0$ but as $\lambda\neq 0$ this means that $\alpha=0$.
\end{proof}

\begin{rmk}
Combining Lemma \ref{lem:hk.ALE.sd}, Proposition \ref{prop:hk.ALE.1forms} and Lemma \ref{lem:hk.ALE.asd}, together with the fact that the Hodge star and $\Delta$ commute, yields Theorem \ref{thm:grav.inst}(b).
\end{rmk}

\section{Functions on ALF gravitational instantons}\label{sec:funct.ALF}

In this section we prove Theorem \ref{thm:An.ALF}, i.e.~that there are no $L^2$ eigenfunctions on $A_{n}$ ALF gravitational instantons (also known as multi-Taub--NUT spaces).  A key tool will be a Fourier decomposition and so, for ease of exposition, we will assume that our eigenfunctions are complex-valued.

\subsection{Multi-Taub--NUT}

We briefly recall the description of $A_{n}$ ALF gravitational instantons, equivalently multi-Taub--NUT spaces, via the Gibbons--Hawking ansatz \cite{GibbonsHawking,Minerbe2011}.  

Let $x_0,\ldots,x_n\in\R^3$ be distinct and let $m>0$.  Define
\begin{equation}\label{eq:V}
        V(x)=m+\sum_{j=0}^n\frac{1}{2|x-x_j|},
        \qquad
        x\in X_0=\R^3\setminus\{x_0,\ldots,x_n\}.
\end{equation}
Then there is a principal $\mathrm{U}(1)$-bundle $\pi:M_0\to X_0$ with connection 1-form $\eta$ satisfying
\begin{equation}\label{eq:curvature}
        \dd\eta=-\pi^*(*_{\R^3}\dd V).
\end{equation}
The metric on $M_0$ given by
\begin{equation}\label{eq:metric}
        g=V^{-1}\eta^2+V g_{\R^3},
\end{equation}
 where $g_{\R^3}$ is the Euclidean metric on $\R^3$, 
extends after adding one point above each centre $x_j\in \R^3$ to a   complete hyperk\"ahler metric on $M\cong M_0\cup\{x_0,\ldots,x_n\}$. 

\begin{dfn}\label{dfn:An.ALF}
We define $(M,g)$ to be the \emph{multi-Taub--NUT} space or \emph{$A_{n}$ ALF gravitational instanton} with mass $m>0$ and centres $x_0,\ldots,x_n$.    
\end{dfn}

\begin{eg}
The Taub--NUT metric is defined on $\R^4$, which arises from taking $n=1$ (and $x_1=0$ without loss of generality) in Definition \ref{dfn:An.ALF}.
\end{eg}

\begin{rmk}  Throughout the rest of the section we will assume that we are working on an $A_{n}$ ALF gravitational instanton $(M,g)$ as given by Definition \ref{dfn:An.ALF} and use the notation above, particularly $V$ and $\eta$.  Notice that rescaling $(M,g)$ will scale the centres $x_j$ and the mass $m$.  Therefore, by scaling we can assume that  $m=1$, which we do from now on.
\end{rmk}

For later use, we note that the volume form of $g$ is:
\begin{equation}\label{eq:vol.V}
\vol_g=V\eta\wedge\vol_{\R^3}.
\end{equation}

\begin{rmk}
The Atiyah--Hitchin metric and its double cover are asymptotic to a Taub--NUT geometry with ``negative mass'' in the sense that, in the notation above, the corresponding Gibbons--Hawking potential has the asymptotic form $V=1-\frac{k}{2|x|}$ for some $k>0$.  It is this negative mass behaviour at infinity which seems to be instrumental in allowing for the existence of non-zero $L^2$ eigenfunctions, cf.~\cite{Schroers3, GibbonsManton}.
\end{rmk}

\subsection{Fourier decomposition and the reduced equation}

The multi-Taub--NUT space $(M,g)$ admits a circle action  generated by a vertical (for the fibration $\pi:M\to\R^3$) vector field $\xi$ dual to $\eta$ (namely, $\eta(\xi) = 1)$. In local fibre coordinates $\psi$ of period $2\pi$ we have that $\xi=\partial_\psi$. 

We can then decompose any complex-valued $L^2$ function $f$ on $(M,g)$ into (vertical) Fourier modes $f_s$ indexed by $s\in\Z$, which then satisfy:
\begin{equation}\label{eq:circle.weight}
    -\ii\Lie_\xi f_s=sf_s.
\end{equation}

The circle action generated by $\xi$ is isometric, so commutes with the Laplacian $\Delta$.  Hence, for all $\lambda\in\R$ we have an orthogonal decomposition 
\begin{equation}\label{eq:E0s}
    \mathcal{E}^0(\lambda)=\widehat{\bigoplus}_{s\in\Z}\mathcal{E}^0_s(\lambda)
\end{equation}
where each $f_s\in\mathcal{E}^0_s(\lambda)$ satisfies \eqref{eq:circle.weight}.  We now wish to dimensionally reduce the eigenvalue equation satisfied by $f_s$ to an equation on the subset $X_0$ of $\R^3$.

\begin{rmk}
    The following dimensional reduction using the circle action is not possible for $D_n$ type ALF gravitational instantons as they do not admit such a global action.  However, Atiyah--Hitchin and its double cover have circle actions at infinity, and there are $D_n$ ALF metrics with \emph{approximate} circle actions at infinity since they are asymptotic to $A_n$ ALF metrics  (cf.~\cite{SchroersSinger}), so some sort of dimensional reduction at infinity may be possible in these settings. 
\end{rmk}

In a local trivialisation of $\pi:M_0\to X_0$ we can find a 1-form $A$ on (an open subset of) $X_0$ so that
\begin{equation}\label{eq:localA}
        \eta=\dd\psi+A,
        \qquad
        \dd A=-*_{\R^3}\dd V.
\end{equation}
A function $f_s$ satisfying \eqref{eq:circle.weight} can be written in this trivialisation as
\begin{equation}\label{eq:fs.u}
        f_s(x,\psi)=e^{\ii s\psi}u(x).
\end{equation}
Globally, $u$ is not necessarily a scalar function, but it is a section of a   Hermitian line bundle $L^s\to X_0$ associated to the circle action with weight $s$: specifically, $L^s=L^{\otimes s}$ where $L$ is the line bundle associated  to the $\mathrm{U}(1)$-bundle $\pi:M_0\to X_0$.  (We take the usual convention that $L^0$ is the trivial line bundle and choosing negative $s$ amounts to taking tensor powers of the line bundle dual to $L$.)  If we set
\begin{equation}\label{eq:P}
        P=-\ii\nabla_{\R^3}
\end{equation}
then we see that the local formula
\begin{equation}\label{eq:mag.mom}
    P_s=P-sA
\end{equation}
 is gauge-covariant and hence defines a global operator on sections of $L^s$.  In fact, $P_s=-\ii\nabla_s$ where $\nabla_s$ is the covariant derivative on $L^s$.

 \begin{dfn}\label{dfn:Ps} Recall $X_0$ as in \eqref{eq:V} and let $L^s\to X_0$ be the Hermitian line bundle associated to the circle action defined by $\xi$ with weight $s$ as above.  We then let \begin{equation}
     P_s=-\ii\nabla_s:\Gamma(L^s)\to\Gamma(T^*X_0\otimes L^s)
     \end{equation} 
     be the \emph{magnetic momentum} operator, which is given by \eqref{eq:mag.mom} in a given local trivialisation of $\pi:M_0\to X_0$.  We also let $P_s^*$ denote the formal adjoint of $P_s$.
 \end{dfn}

With this definition in hand, we can derive our reduced equation for eigenfunctions $f_s\in\mathcal{E}^0_s(\lambda)$.

\begin{lem}
\label{lem:reduction}  Recall the notation of \eqref{eq:E0s} and Definition \ref{dfn:Ps}. 
Let $f_s\in\mathcal{E}^0_s(\lambda)$ and let $u\in\Gamma(L^s)$ be the associated section to $f_s$ via \eqref{eq:fs.u}. Then $u$ satisfies  
\begin{equation}\label{eq:raw}
        \bigl(P_s^*P_s+s^2V^2-\lambda V\bigr)u=0.
\end{equation}

\end{lem}

\begin{proof}
Let $f_s$ satisfy \eqref{eq:circle.weight}.  Then $f_s$ is smooth by Lemma \ref{lem:smooth}.  Suppose that $f_s$ is compactly supported in $M_0$ in one local trivialisation of $\pi:M_0\to X_0$, where we use the notation \eqref{eq:localA}. Then we can  write $f_s(x,\psi)=e^{\ii s\psi}u(x)$.      
 Using the inverse of the metric \eqref{eq:metric}, formula \eqref{eq:mag.mom} and this notation gives
\begin{equation}
        |\dd f_s|_g^2
        =V^{-1}|(\nabla_{\R^3}-\ii sA)u|_{\R^3}^2+Vs^2|u|_{\R^3}^2.
\end{equation}
Using \eqref{eq:vol.V}, the volume form of $g$ in this trivialisation is 
\begin{equation}
    \vol_g=V\dd\psi\wedge\vol_{\R^3},
\end{equation}  
so integrating over the fibres and using \eqref{eq:mag.mom} gives   
\begin{align}
        \norm{f_s}_{L^2}^2
        &=2\pi\int_{X_0}V|u|_{\R^3}^2\vol_{\R^3},\label{eq:norm}\\
        \ip{f_s}{\Delta f_s}_{L^2}=\norm{\d f_s}_{L^2}^2
        &=2\pi\int_{X_0}\bigl(|P_su|_{\R^3}^2+s^2V^2|u|_{\R^3}^2\bigr)\vol_{\R^3}.\label{eq:form}
\end{align}
Equation \eqref{eq:raw} follows for $f_s$ supported in one trivialisation by testing against sections compactly supported in the same trivialisation.  We can then use gauge covariance and the fact that any eigenfunction $f_s$ is smooth to patch together to obtain the equation \eqref{eq:raw} globally.
\end{proof}

Note  the elementary algebraic identity: 
\begin{equation}
      s^2V^2-\lambda V
        =\bigl(s^2(V^2-1)-\lambda(V-1)\bigr)-(\lambda-s^2).  
\end{equation}
If we then let
    \begin{align}
\label{eq:h}
 h(x)&=V(x)-1=\sum_{j=0}^n\frac{1}{2|x-x_j|},\qquad\quad\! \text{for $x\in X_0$,}\\
\label{eq:E}
        E_{s,\lambda}&=\lambda-s^2,\\
\label{eq:W}
        W_{s,\lambda}&=s^2(V^2-1)-\lambda(V-1)=(2s^2-\lambda)h+s^2h^2,\\
\label{eq:Hs}   H_{s,\lambda}&=P_s^*P_s+W_{s,\lambda},
\end{align}
we obtain the following useful result, which gives a reformulation of \eqref{eq:raw}.

\begin{lem} We have that $u\in\Gamma(L^s)$ satisfies \eqref{eq:raw} if and only if
\begin{equation}\label{eq:reduced}
        H_{s,\lambda}u=E_{s,\lambda}u.
\end{equation}
\end{lem}

\begin{rmk}
    Equation \eqref{eq:reduced} is of the form of an eigenvalue equation for a magnetic Schr\"odinger operator $H_{s,\lambda}$, with potential $W_{s,\lambda}$.  The solution $u$ has associated eigenvalue (or energy) $E_{s,\lambda}$.   
\end{rmk}

\subsection{Non-positive energy}

As we indicated above, we have viewed \eqref{eq:reduced} as an eigenvalue equation for a magnetic Schr\"odinger operator, whose eigenvalue or energy is $E_{s,\lambda}$.  In line with its interpretation as an energy, we would expect that there will be no non-trivial solutions with non-positive energy.  We now show that this is the case.

\begin{lem}\label{lem:non.pos.Esl}
 Recall   \eqref{eq:E0s} and \eqref{eq:E}.   We have that $\mathcal{E}^0_{s}(\lambda)=\{0\}$ whenever $E_{s,\lambda}=\lambda-s^2\leq 0$.
\end{lem}

\begin{proof} Let $f_s\in\mathcal{E}^0_s(\lambda)$ and let $u$ be the corresponding section of $L^s$ as in \eqref{eq:fs.u}.

Using \eqref{eq:norm}--\eqref{eq:form} in \eqref{eq:raw} gives
\begin{equation}\label{eq:positive-identity}
        \int_{X_0}|P_su|_{\R^3}^2\vol_{\R^3}
        +\int_{X_0}V(s^2V-\lambda)|u|_{\R^3}^2\vol_{\R^3}=0.
\end{equation}
Since $V>1$ by \eqref{eq:V} (recall that $m=1$) and $\lambda\leq s^2$ by assumption, the second integrand is non-negative.  Hence both terms in \eqref{eq:positive-identity} vanish.

If $s\neq 0$ then $s^2V-\lambda>s^2-\lambda\geq 0$, so $u=0$.  We also confirm that if $\lambda<0$ (which we already know is impossible by Lemma \ref{lem:integrability.grav.inst}) then $s^2V-\lambda>0$, which again forces $u=0$.

 The only remaining possibility is $s=0$ and $\lambda=0$.  Then \eqref{eq:positive-identity} says that $u$ is constant on $X_0$, but then $f_s$ is also constant (as $s=0$), which means that $f_s=0$ as $M$ has infinite volume and $f_s\in L^2$ by assumption.
 \end{proof}

 \begin{rmk}
It is this setting of non-positive energy where it appears crucial to have a  circle action on the gravitational instanton, since the argument is global. On Atiyah--Hitchin and its double cover the known $L^2$ eigenvalues \cite{Schroers3,GibbonsManton} correspond to negative energy in our notation.   
 \end{rmk}

 \subsection{Positive energy: setup}\label{ss:pos.set}

 Assume now that $f_s\in\mathcal{E}^0_{s}(\lambda)$ with
\begin{equation}
        E_{s,\lambda}=\lambda-s^2>0.
\end{equation}
 This situation is more challenging and requires us to build on existing analytic machinery.  Specifically, we extend parts of the theory for magnetic Schr\"odinger operators developed in \cite{AHK}.  The idea is to show that $f_s$ vanishes identically outside a compact set, but then unique continuation will imply that $f_s$ vanishes everywhere (as we saw in the proof of Proposition \ref{prop:noL^2.eigenforms.flat}).

\begin{rmk}
    As in the proof of Proposition \ref{prop:noL^2.eigenforms.flat}, one might be able to use methods from the scattering calculus in this setting to rule out positive eigenvalues, but we decide to take a more elementary approach as in \cite{AHK}.
\end{rmk}
 
For $R>0$ we let 
\begin{equation}\label{eq:DR}
    B_R=\{x\in\R^3\,:\,|x|<R\} \quad\text{and}\quad \mathcal{D}_R =\{x\in\R^3\,:\,|x|> R\}.
\end{equation}
Recall the points $x_0,\ldots,x_n\in\R^3$ in \eqref{eq:V} and choose $R_0>0$ sufficiently large that $x_j\in B_{R_0/2}$ for all $j$.  Then on $\mathcal{D}_{R_0}\subseteq X_0$ the line bundle $L^s$ (given in Definition \ref{dfn:Ps}) and its connection $\nabla_s$ are smooth.   

Let $u$ be the section of $L^s$ associated to $f_s$ as in \eqref{eq:fs.u}. Since $V\to1$ as $|x|\to \infty$ by \eqref{eq:V}, equation \eqref{eq:norm} implies 
\begin{equation} u\in L^2(\mathcal{D}_{R_0};L^s)
\end{equation}
(where here and in the following this means with respect to the Euclidean metric). The identity \eqref{eq:form} implies the magnetic energy of $u$ on $X_0$ (and hence $\mathcal{D}_{R_0}$) is finite:
\begin{equation}\label{eq:mag.finite}
    \int_{\mathcal{D}_{R_0}}(|P_su|^2_{\R^3}+s^2V^2|u|_{\R^3}^2)\vol_{\R^3}<\infty.
\end{equation}

We let $X$ be the radial vector field on $\R^3$, i.e.~for $r=|x|$ we have
\begin{equation}\label{eq:radialvf}
    X(x)=r\partial_r.
\end{equation}
This will play an important role in applying the theory of \cite{AHK}.

\subsection{Magnetic term}

For the theory of \cite{AHK} one needs control on the various terms appearing in our magnetic Schr\"odinger operator $H_{s,\lambda}$ in \eqref{eq:Hs}.  The first concerns the so-called \emph{magnetic field}, which is related to the curvature of the connection on $L^s$.

By \eqref{eq:localA} and \eqref{eq:mag.mom}, the curvature of the connection on $L^s$ is 
\begin{equation}
    -\ii s\dd\eta=\ii s*\dd V.
\end{equation}  We  need to analyse the asymptotic behaviour of $\iota_X(s\dd\eta)$ (the ``magnetic term''), where $X$ is as in \eqref{eq:radialvf}. All norms will be calculated with respect to the Euclidean metric on $\R^3$.  

Notice that
\begin{equation}
    |\iota_X\dd\eta|=|\iota_X*\dd V|=|x\wedge \d V|.
\end{equation}
Then by \eqref{eq:V} we have
\begin{equation}\label{eq:Bs}
             s\nabla V(x)
        =-s\sum_{j=0}^n \frac{x-x_j}{2|x-x_j|^3},
\end{equation}  
which can be interpreted as the ``magnetic field''.    
We deduce that
\begin{align}\label{eq:cross}
        sx\times \nabla V(x) 
        &=s\sum_{j=0}^n  \frac{x\times x_j}{2|x-x_j|^3}.
\end{align}
Since $x_j\in B_{R_0/2}$ for all $j$, we have $|x-x_j|\geq |x|/2$  on $\mathcal{D}_{R_0}$ for all $j$, and hence
\begin{equation}\label{eq:beta-zero}
    |\iota_X(s\dd\eta)|=    |sx\times \nabla V(x)|=O(|x|^{-2})\quad\text{as $|x|\to\infty$.}
\end{equation}
This is the  asymptotic decay on the magnetic term we require.

\subsection{Electric potential}

Another important ingredient (following \cite{AHK})  will be to control the asymptotic behaviour of the (electric) potential $W_{s,\lambda}$ in \eqref{eq:W}.

Recall $h$ in \eqref{eq:h}.   Since the centres $x_j$ remain in a fixed compact set, we   see that
\begin{equation}\label{eq:hdecay1}
        h(x)=O(|x|^{-1})\quad\text{as $|x|\to\infty$.}
\end{equation}
Moreover, recalling $X$ in \eqref{eq:radialvf}, 
\begin{equation}
      Xh(x)=  x\cdot\nabla h(x)
        =-\sum_{j=0}^n \frac{x\cdot(x-x_j)}{2|x-x_j|^3}
        =-h(x)+O(|x|^{-2})\quad\text{as $|x|\to\infty$.}
\end{equation}

Consequently, $W_{s,\lambda}$ in \eqref{eq:W} satisfies
\begin{equation}\label{eq:Wdecay}
        W_{s,\lambda}(x)=O(|x|^{-1})\quad\text{as $|x|\to\infty$,}
\end{equation}
and
\begin{align}\label{eq:XWdecay}
     XW_{s,\lambda}(x) 
        &=(2s^2-\lambda)x\cdot\nabla h(x)
          +2s^2h(x)\,x\cdot\nabla h(x)     =O(|x|^{-1})\quad\text{as $|x|\to\infty$.}
\end{align}
As we shall see, these are the required asymptotic conditions on the potential.

\subsection{Positive energy: vanishing theorem}

The key technical tool we will need to prove the non-existence of positive eigenvalues for \eqref{eq:reduced}, and hence complete the proof of Theorem \ref{thm:An.ALF}, is the following vanishing theorem on exterior domains as in \eqref{eq:DR}.

\begin{thm}\label{thm:An.vanish}
 Recall Definition \ref{dfn:Ps} and let $u\in L^2(\mathcal{D}_{R_0};L^s)$ be a solution to \eqref{eq:reduced} (i.e.~$H_{s,\lambda}u=E_{s,\lambda}u$) with finite  magnetic energy (i.e.~satisfying \eqref{eq:mag.finite}).  
If $E_{s,\lambda}>0$, then there exists $R_1>R_0$ such that
\begin{equation}
        u=0\quad\text{on }\mathcal{D}_{R_1}.
\end{equation}   
\end{thm}

\begin{proof}
   This result is an analogue of \cite[Theorem 4.8]{AHK} and we follow its proof in \cite[Sections 3 and 4]{AHK}, recording any notable changes.  The proof is quite lengthy and so we break it up into several steps.

\paragraph{\bf{Step 1: radial gauge.}}  
In \cite{AHK} a global potential $A$ (as appears in \eqref{eq:localA} and \eqref{eq:mag.mom}) and Poincaré gauge (cf.~\cite[Section 2.2]{AHK}) for $A$ are used. The key difference here is that $L^s$ will typically be non-trivial over the sphere at infinity in $X_0$, so one cannot choose a global potential.  Radial parallel transport allows us to find the required substitute, which enables us to define a suitable gauge-fixing.

For $x\in\mathcal{D}_{R_0}$ and $t$ such that $e^t x\in\mathcal{D}_{R_0}$, let
\begin{equation}\label{eq:tau.t.x}
        \tau_{t,x}:L^s_{e^t x}\longrightarrow L^s_x
\end{equation}
be parallel transport along the radial segment from $e^tx$ to $x$ using the connection (locally given by $sA$) on $L^s$.   Since the connection is unitary, $\tau_{t,x}$ is unitary.   

We can now perform our required radial gauge-fixing.

\begin{lem} 
\label{lem:radial-gauge}
For any $R>R_0$, radial parallel transport identifies $L^s$ with the pullback of $L^s|_{S_R}$, where $S_R$ is the sphere of radius $R$ in $\R^3$.  With this identification, in every radially parallel local unitary frame with connection form $sA$, the radial vector field $X$ in \eqref{eq:radialvf} satisfies  
\begin{equation}\label{eq:radial-gauge}
        sA(X)=0.
\end{equation}
\end{lem}

\begin{proof}
The first statement about radial parallel transport identifying $L^s$ with $L^s|_{S_R}$ is clear.  Then, a local frame on $\mathcal{D}_{R_0}$ obtained by transporting a local frame on $S_R$ along radial segments is parallel in the radial direction. Thus the covariant derivative in the $X$ direction is ordinary differentiation, which is equivalent to \eqref{eq:radial-gauge}.  This construction is local 
on $S_R$ but intrinsic along each ray, so it patches even if $L^s|_{S_R}$ is topologically non-trivial.
\end{proof}

\paragraph{\bf Step 2: dilations.}  Recall the parallel transport maps in \eqref{eq:tau.t.x}.  
Define, initially on compactly supported sections $v$ of $L^s$ over $\mathcal{D}_{R_0}$,
\begin{equation}\label{eq:Ut}
        (U_tv)(x)=e^{3t/2}\tau_{t,x}v(e^t x).
\end{equation}
This extends to a unitary operator $U_t$ on $L^2$ whenever both $v$ and $U_tv$ are supported in $\mathcal{D}_{R_0}$.  We now describe how $U_t$ interacts with the magnetic momentum operator $P_s$.

\begin{lem}
\label{lem:dilation-covariance}
Let $v\in C_c^\infty(\mathcal{D}_{R_0};L^s)$ and recall the operator $P_s$ in \eqref{eq:mag.mom}. In radial gauge (as in Lemma \ref{lem:radial-gauge}),
\begin{equation}\label{eq:covariance}
        P_s U_tv=e^tU_t P_s v+U_tG^s_tv,
\end{equation}
where $G^s_t$ is a real $1$-form satisfying
\begin{equation}\label{eq:Gt-limit}
        \frac1tG_t^sv\longrightarrow s(\iota_X\dd\eta) v
        \quad\text{in }L^2\text{ as }t\to0.
\end{equation}
\end{lem}

\begin{proof}
We work in a radially parallel local unitary frame, so that \eqref{eq:mag.mom} holds with $sA(X)=0$ by Lemma \ref{lem:radial-gauge}.  In this frame $U_t$ is the usual scalar dilation operator.  The calculation leading to \cite[(3.11)]{AHK} gives \eqref{eq:covariance}, 
with $G^s_t$ measuring the difference between $sA$ and its rescaling under dilation by $t$, which is a globally defined real 1-form.  Moreover,  
\begin{equation}
    \lim_{t\to 0}\frac{G^s_t}{t}=s\mathcal{L}_XA.
\end{equation}
Using Cartan's formula and $sA(X)=0$ gives (recalling \eqref{eq:localA})
\begin{equation}
      s\Lie_XA=\iota_X(s\dd A)+\dd(sA(X))=s\iota_X\dd\eta.
\end{equation}     
This proves \eqref{eq:Gt-limit}. 
\end{proof}

\begin{rmk}\label{rmk:tildeB} Lemma \ref{lem:dilation-covariance} shows that one can modify the work in \cite{AHK}  
by replacing the term $\tilde{B}$ there (see e.g.~\cite[Proposition 3.6]{AHK}) with $s\iota_X\dd\eta$.
\end{rmk}

\paragraph{\bf Step 3: magnetic virial identity.} Recall $W_{s,\lambda}$ in \eqref{eq:W} and $H_{s,\lambda}$ in \eqref{eq:Hs}.  Let $\ii D$ be the generator of the dilation group \eqref{eq:Ut}, i.e.
\begin{equation}\label{eq:generator}
\ii D=\frac{dU_t}{dt}|_{t=0}.    
\end{equation}
For compactly supported  sections $v$ of $L^s$ over $\mathcal{D}_{R_0}$, we may differentiate the inner product $\ip{v}{U_t^{-1}H_{s,\lambda}U_tv}$ at $t=0$ and use Lemma \ref{lem:dilation-covariance} to obtain
\begin{equation}\label{eq:basic-virial}
        \ip{v}{\ii[H_{s,\lambda},D]v}
        =2\norm{P_sv}^2
         +2s\operatorname{Re}\ip{(\iota_X\dd\eta)v}{P_sv}
         -\ip{v}{(XW_{s,\lambda})v}.
\end{equation}
This identity appears in the ``magnetic virial theorem''  \cite[Theorem 3.8]{AHK}, with $\tilde{B}$ there replaced by $s\iota_X\dd\eta$ as expected by Remark \ref{rmk:tildeB}.  The replacement is justified by Lemma \ref{lem:dilation-covariance}; all other terms in the derivation are scalar and the argument for them uses only the Leibniz rule for the connection on $L^s$ and so does not see whether $L^s$ is  trivial over the sphere at infinity or not. 

\paragraph{\bf Step 4: weights.}  The next idea in \cite[Section 3]{AHK} is to multiply the sections of $L^s$ by suitable ``weights'' to deduce decay.  For these weights, let  $F:\mathcal{D}_{R_0}\to \R$ be a bounded smooth radial function satisfying the hypotheses of \cite[Lemma 3.16]{AHK}, and let $v$ be a compactly supported section of $L^s$ over $\mathcal{D}_{R_0}$.  Since
\begin{equation}\label{eq:vF}
        P_s (e^Fv)=e^FP_s v-\ii e^F(\dd F)v,
\end{equation}
the exponentially weighted virial identities in \cite[Section 3.4]{AHK}, which extend \eqref{eq:basic-virial}, apply verbatim to bundle-valued sections.  Metric compatibility of the connection on $L^s$ replaces ordinary integration by parts there.  The only ``magnetic term'' is again the one already identified in \eqref{eq:basic-virial} (cf.~Remark \ref{rmk:tildeB}).  

Thus \cite[Lemmas 3.16, 3.18 and 3.19]{AHK} hold on $\mathcal{D}_{R_0}$, with $\tilde{B}$ there replaced by $s\iota_X\dd\eta$, where cutoff functions are used to avoid the inner boundary on $\mathcal{D}_{R_0}$ and to truncate sections at large radius. We now provide some more detail on these cutoff arguments, which will also be useful for the remaining steps.

Let us first discuss the inner cutoff. Choose a smooth radial function $\chi_0:\R^3\to [0,1]$ such that
\begin{equation}
    \chi_0(x)=\left\{\begin{array}{cl} 0 & |x|\leq R_0+1,\\ 
    1 & |x|\geq R_0+2. \end{array}\right.
\end{equation}
 Recall $X$ in \eqref{eq:radialvf} and  $D$ in \eqref{eq:generator}.  Put
\begin{equation}\label{eq:cutoff.inner}
    \widetilde{X}=\chi_0 X,
\qquad
\mathcal{A}_0=\{x\in\R^3\,:\,R_0+1\leq |x|\leq R_0+2\}.
\end{equation}
Note the annular region $\mathcal{A}_0$ is compactly contained in $\mathcal{D}_{R_0}$. The flow of $\widetilde X$, lifted to $L^s$ using parallel transport as in \eqref{eq:Ut}, defines a unitary group on $L^2 \left(\mathcal{D}_{R_0}; L^s\right)$ with generator
\begin{equation}\label{eq:tilde.D}
    \ii\widetilde D
       =\nabla^s_{\widetilde X}
        +\frac12\operatorname{div}\widetilde X.
\end{equation}
 Since $\widetilde D = D$   on $\mathcal{D}_{R_0+2}$ and $\widetilde D = 0$ on $\mathcal{D}_{R_0} \setminus \mathcal{D}_{R_0+1}$, the terms on $\mathcal{D}_{R_0+2}$ arising from commutators with $\tilde{D}$ coincide with those obtained using $D$, while all additional terms in $\mathcal{D}_{R_0}\setminus\mathcal{D}_{R_0+2}$ are supported in $\mathcal{A}_0$.

 Therefore, there exists a constant $C_0>0$ such that, for any $L^2$ solution $u$ to \eqref{eq:reduced} in $\mathcal{D}_{R_0}$ and weight $F:\mathcal{D}_{R_0}\to\R$, 
 the contribution $\mathcal{R}_Fu$ of the terms supported in $\mathcal{A}_0$ arising from the inner cutoff satisfy
\begin{equation}
    |\mathcal R_Fu|
    \leq C_0\left(
       \norm{P_s(e^Fu)}_{L^2(\mathcal{A}_0)}^2
       +\norm{e^Fu}_{L^2(\mathcal{A}_0)}^2
    \right).
\end{equation}
Using \eqref{eq:mag.finite} and \eqref{eq:vF}, we obtain
\begin{equation} \label{eq:fixed-remainder}
    |\mathcal R_Fu| \leq C_0(u) \left(1+\|dF\|_{C^0(\mathcal{A}_0)}^2\right) e^{2\sup_{\mathcal{A}_0}\! F}
\end{equation}
for $C_0(u)>0$ which may now depend on $u$, but it is independent of $F$.

Now let us discuss the outer cutoff, needed to justify the identities at infinity. Given $T>R_0+2$ choose a smooth radial function $\chi_T:\R^3\to [0,1]$ such that
\begin{equation}
\chi_T(x)=\left\{\begin{array}{cl} 1 & |x|\leq T, \\
0 & |x|\geq 2T.\end{array}\right.
\end{equation}
If we then let
\begin{equation}\label{eq:outer.annulus}
    \mathcal{A}_T=\{x\in\R^3\,:\,T\leq|x|\leq 2T\},
\end{equation}
which is again compact, we have on $\mathcal{A}_T$ that
\begin{equation}\label{eq:T.est}
    |\dd\chi_T|=O(T^{-1}),\qquad
    |\nabla^2\chi_T|=O(T^{-2}),\qquad
    |X|=O(T).
\end{equation}
We can then perform a similar cutoff using $\chi_T$ to the radial vector field as in \eqref{eq:cutoff.inner} and obtain a modified generator to the dilation action as in \eqref{eq:tilde.D}.  This then yields error terms arising in the commutator identities supported in $\mathcal{A}_T$.

Therefore, for each fixed bounded weight $F:\mathcal{D}_{R_0}\to\R$ and $T>R_0+2$, there exists a constant $C_F(T)>0$ so that the error term $\mathcal{E}_{T,F}u$ supported in $\mathcal{A}_T$ for any $L^2$ solution $u$ to \eqref{eq:reduced} on $\mathcal{D}_{R_0}$ satisfies  
\begin{equation} 
    |\mathcal E_{T,F}u|
    \leq C_F(T)\left(
       \norm{P_s(e^Fu)}_{L^2(\mathcal{A}_T)}^2
       +\norm{e^Fu}_{L^2(\mathcal{A}_T)}^2
    \right).
\end{equation}
By \eqref{eq:T.est} the constant $C_F(T)$ (for $T$ fixed) tends to zero as $T\to\infty$, so we deduce from \eqref{eq:mag.finite} and the fact that $F$ is fixed that
 \begin{equation}\label{eq:escaping-remainder}
    |\mathcal E_{T,F}u|\to 0\quad\text{as $T\to\infty$.}
 \end{equation}

\paragraph{\bf Step 5: fast decay.}  Recall \eqref{eq:E} and \eqref{eq:Hs} and that $u$ satisfies $H_{s,\lambda}u=E_{s,\lambda}u$ with $E_{s,\lambda}>0$. Recalling \eqref{eq:DR} we let 
\begin{equation}\label{eq:mu*}
        \mu_*=\sup\{\mu\geq0: e^{\mu \sqrt{1+|x|^2}}u\in L^2(\mathcal{D}_R;L^s)\text{ for every }R>R_0\}.
\end{equation}
Consider the  weights (for $\mu>0$ and $\varepsilon>0$)
\begin{equation}
        F_{\mu,\varepsilon}(x)=\frac{\mu}{\varepsilon}\left(1-e^{-\varepsilon \sqrt{1+|x|^2}}\right).
\end{equation}
They satisfy the bounds
\begin{equation}\label{eq:Fbounds}
    0\leq F_{\mu,\varepsilon}(x) \leq \mu\sqrt{1+|x|^2},
\qquad
|dF_{\mu,\varepsilon}(x)|\leq\mu.
\end{equation}
Using these weights in the previous step together with the estimates \eqref{eq:beta-zero}, \eqref{eq:Wdecay} and \eqref{eq:XWdecay}, the proof of \cite[Proposition 4.1]{AHK} goes through to show that $\mu_*$ cannot be finite.  (Note that, in the notation of \cite[Proposition 4.1]{AHK}, we have $\Lambda=0$ in this case.)  

The idea is to suppose otherwise, i.e.~$\mu_*$ is finite. For the fixed radius $R=R_0+2$, one finds decreasing sequences $\mu_k\to \mu_*$  and $\varepsilon_k\to 0$ such that, if we let 
\begin{equation}
    u_k=e^{F_{\mu_k,\varepsilon_k}}u
\end{equation}
then $\|u_k\|_{L^2(\mathcal{D}_{R})}\to \infty$ as $k\to\infty$.  We set
\begin{equation}
    v_k=\frac{u_k}{\|u_k\|_{L^2(\mathcal{D}_{R })} }.
\end{equation}

For each $k$, we first let the outer cutoff radius $T$ tend to
infinity, so the escaping annulus remainder \eqref{eq:escaping-remainder} tends to zero. The bounds \eqref{eq:Fbounds}, together with the boundedness of $(\mu_k)$ and the compactness of $\mathcal{A}_0$ in \eqref{eq:cutoff.inner}, imply that the fixed-annulus remainder $\mathcal{R}_{\mu_k, \varepsilon_k}u_k$ from \eqref{eq:fixed-remainder} is uniformly bounded in \(k\):
\begin{equation}
|\mathcal R_{\mu_k,\varepsilon_k}u_k|
\leq C_0(u).
\end{equation}
Consequently,
\begin{equation}
\frac{|\mathcal R_{\mu_k,\varepsilon_k}u_k|}
{\|u_k\|_{L^2(D_R)}^2}
\leq
\frac{C_0(u)}
{\|u_k\|_{L^2(D_R)}^2}
\to0\quad \text{as $k\to\infty$.}
\end{equation}
Therefore, one finds (cf.~\cite[(4.16)]{AHK}) the following upper estimate: 
\begin{equation}\label{eq:bootstrap-upper}
        \limsup_k \ip{v_k}{\ii[H_{s,\lambda},\widetilde D]v_k}\leq0.
\end{equation}

However, one also proves (cf.~\cite[(4.17)]{AHK}) a lower estimate: 
\begin{equation}\label{eq:bootstrap-lower}
        \liminf_k \ip{v_k}{\ii[H_{s,\lambda},\widetilde D]v_k}
        \geq 2(E_{s,\lambda}+\mu_*^2)>0.
\end{equation}
(Note that the constants $\beta$, $\omega_1$, $\omega_2$ in \cite[(4.17)]{AHK} are all zero in our setting.) Equations \eqref{eq:bootstrap-upper}--\eqref{eq:bootstrap-lower} provide the required contradiction.  

The proof of the two estimates \eqref{eq:bootstrap-upper}--\eqref{eq:bootstrap-lower} uses only the weighted identities from the previous step, Cauchy--Schwarz, the IMS localisation formula, and  the asymptotic behaviour \eqref{eq:beta-zero}, \eqref{eq:Wdecay} and \eqref{eq:XWdecay}.  The IMS localisation formula is valid for any Hermitian connection on $L^s$ as it only uses the Leibniz rule, so does not involve the topology of $L^s$ over the sphere at infinity.

Since $\mu_*=+\infty$ in \eqref{eq:mu*}  we deduce, as in \cite{AHK}, that we have fast (specifically, superexponential) decay on the end:
\begin{equation}\label{eq:superexp}
        e^{\mu \sqrt{1+|x|^2}}u\in L^2(\mathcal{D}_{R};L^s)
        \quad\text{for every }\mu>0\text{ and all }R>R_0.
\end{equation}

\paragraph {\bf Step 6: conclusion.}
The final step in the proof, as in \cite[Theorem 4.8]{AHK}, is to multiply our eigensection $u$ by the weights (defined for $\mu,\varepsilon,\sigma>0$):
\begin{equation}\label{eq:F.mu.eps.sig}
        F_{\mu,\varepsilon,\sigma}(x)=\frac{\mu}{\varepsilon}
        \left(1-e^{-\varepsilon\sqrt{\sigma+|x|^2}}\right).
\end{equation}
Since \eqref{eq:superexp} holds, these
weights are admissible on the end. Recalling \eqref{eq:cutoff.inner},  for \(0<\sigma\leq1\), we get the bounds
\begin{equation}
0\leq F_{\mu,\varepsilon,\sigma}(x) \leq \mu\sqrt{\sigma+|x|^2}\leq\mu (R_0+3)
\quad\text{on }\mathcal{A}_0,
\qquad
|dF_{\mu,\varepsilon,\sigma}(x)|
\leq
\mu.
\end{equation} Using \eqref{eq:fixed-remainder},
the fixed-annulus remainder term, which we write as $\mathcal{R}_{\mu,\varepsilon,\sigma}u$, satisfies
\begin{equation}\label{eq:fixed.annulus.final}
|\mathcal R_{\mu,\varepsilon,\sigma}u| \leq C_0(u)(1+\mu)^2e^{2\mu (R_0+3)},
\end{equation}
where $C_0(u)$ is independent of $\mu$, $\varepsilon$, and $\sigma$.

For fixed $\mu,\varepsilon>0$ and $0<\sigma\leq1$, we insert
the outer cutoff, to generate errors supported in $\mathcal{A}_T$ in \eqref{eq:outer.annulus}.  Letting the radius $T$ tend to infinity,
the escaping-annulus remainders \eqref{eq:escaping-remainder} for the weights \eqref{eq:F.mu.eps.sig} tend to zero. The argument in the proof of
\cite[Theorem~4.8]{AHK} applies, with a minor correction
to the radial computation. Indeed, after letting $\varepsilon\downarrow0$, set
\(\nabla F_{\mu,0,\sigma}=g_{\mu,0,\sigma}x\). It is claimed in \cite[(4.26)]{AHK} that $X^2 g_{\mu, 0, \sigma}$ is negative. Actually, the computation gives
\begin{equation}
X^2g_{\mu,0,\sigma}
=
\frac{\mu |x|^2(|x|^2-2\sigma)}{(\sigma+|x|^2)^{5/2}},
\end{equation}
which need not be non-positive. However, for every \(\delta>0\) and $R>R_0+2$, 
uniformly in \(\mu>0\) and \(0<\sigma\leq 1\), we have
\begin{equation}
\bigl(X^2g_{\mu,0,\sigma}\bigr)_+
\leq
\delta|\nabla F_{\mu,0,\sigma}|^2+C_{\delta,R}
\qquad\text{on } \mathcal{D}_R.
\end{equation}
Choosing \(\delta\) sufficiently small, the gradient term can be
absorbed into the positive term in the weighted inequality, while the remaining constant is inconsequential. Finally, for $\mu$ fixed, we let
$\sigma\to0$, again using the superexponential decay \eqref{eq:superexp}.

The localised version of the calculation in
\cite[(4.24)--(4.28)]{AHK}, together with \eqref{eq:fixed.annulus.final}, therefore gives constants $\kappa,C>0$, independent of
$\mu$, such that, for any fixed radius $R>R_0$  chosen sufficiently large,
\begin{equation}\label{eq:final.est}
    \kappa\mu^2
       \norm{e^{\mu|x|}u}_{L^2(\mathcal D_R)}^2
    \leq
       C\norm{e^{\mu|x|}u}_{L^2(\mathcal D_R)}^2
       +C_0(u)(1+\mu)^2 e^{2\mu (R_0+3)}
\end{equation}
for every $\mu >0$.
Suppose now that $u$ does not vanish on any $\mathcal{D}_R$.  Choose
$R_1>R_0+3$ sufficiently large that \eqref{eq:final.est} holds.  First, we note that
\begin{equation}
    \norm{e^{\mu|x|}u}_{L^2(\mathcal D_{R_1})}^2
    \geq
    e^{2\mu R_1}\norm{u}_{L^2(\mathcal D_{R_1})}^2,
    \qquad
    \norm{u}_{L^2(\mathcal D_{R_1})}^2>0.
\end{equation}
Consequently, for every sufficiently large $\mu$,
\begin{equation}
    (\kappa\mu^2-C)
       \norm{u}_{L^2(\mathcal D_{R_1})}^2
    \leq
       C_0(u)(1+\mu)^2 e^{-2\mu(R_1-R_0-3)}.
\end{equation}
The right-hand side tends to zero as $\mu\to\infty$, whereas the
left-hand side is positive and grows quadratically.  This is a
contradiction.  Hence $u=0$ on $\mathcal D_{R_1}$ for some $R_1>R_0$,
which completes the proof of Theorem~\ref{thm:An.vanish}.
\end{proof}

\begin{rmk}
    As previously discussed, one may be able to obtain a dimensionally reduced equation at infinity (perhaps approximately) for eigenfunctions on some $D_n$ ALF gravitational instantons.  In these cases one may hope for a vanishing result as in Theorem \ref{thm:An.vanish} which can be used to deduce the non-existence of $L^2$ eigenfunctions corresponding to positive energy.  However, this still leaves open the possibility of infinitely many eigenvalues corresponding to negative energy, as is the case for Atiyah--Hitchin and its double cover. 
\end{rmk}

\subsection{Non-existence}

We may now prove Theorem \ref{thm:An.ALF} for eigenfunctions on $A_{n}$ ALF gravitational instantons.

\begin{proof}[Proof of Theorem \ref{thm:An.ALF}] 
To show that $\mathcal{E}^0(\lambda)=0$ for all $\lambda$ it suffices to show that $\mathcal{E}^0_s(\lambda)=0$ for all $s,\lambda$ by \eqref{eq:E0s}.  

Suppose $f_s\in\mathcal{E}^0_s(\lambda)$.  If $\lambda\leq s^2$ then $f_s=0$ by Lemma \ref{lem:non.pos.Esl}.  Suppose instead that $\lambda>s^2$. By our observations in Subsection \ref{ss:pos.set}, we may apply Theorem \ref{thm:An.vanish} to deduce that $f_s$ vanishes outside a compact subset of $M$.  However, $f_s$ satisfies a linear elliptic equation on $M$.  
By unique continuation, $f_s$ must vanish identically.
\end{proof}
 
\section{Eigenforms and special holonomy SU(3), \texorpdfstring{G\textsubscript{2}}{G2} and Spin(7)}\label{sec:special.hol}

In this section we make observations about non-existence of $L^2$ eigenforms in the presence of special holonomy in dimensions $6$, $7$ and $8$.  

As we have already noted (see Corollary \ref{cor:ALF}), in manifolds with special holonomy we can both have examples with $L^2$ eigenfunctions for the Laplacian and without them.  Moreover, if we have such eigenfunctions then we get $L^2$ eigenforms in all degrees (Proposition \ref{prop:functions.forms}), except in the Spin(7) setting (where we only omit 2- and 6-forms).  
As a consequence, we will state our vanishing results under the assumption of no $L^2$ eigenfunctions for the Laplacian.  

In addition, since the special holonomy manifolds we study are Ricci flat, the equality \eqref{eq:Delta.Rough} between the rough and Hodge Laplacians potentially suggests that $\mathcal{E}^1(\lambda)$ may vanish when $\mathcal{E}^0(\lambda)$ vanishes.  We are therefore also interested in settings where there are no $L^2$ eigen-1-forms.
 
\subsection{Decomposition of forms}

We recall that if the holonomy group of a Riemannian $n$-manifold $(M,g)$ is $G$ then we obtain decompositions of the spaces of forms 
\begin{equation}\label{eq:forms.hol.decomp}
    \Omega^k(M)=\sum_j\Omega^k_{p_j}(M),
\end{equation}
where $\sum_jp_j=\binom{n}{k}$ is equal to the total rank of the bundle of $k$-forms, and each $p_j$ denotes the dimension of an irreducible   representation of $G$.   The Hodge star gives isomorphisms $\Omega^k_{p_j}\cong \Omega^{n-k}_{p_j}$ so one needs only to consider $k\leq n/2$.  Moreover, the functions and 1-forms are irreducible, i.e.~do not decompose, so we shall focus on the splitting \eqref{eq:forms.hol.decomp} for $1<k\leq n/2$.

It is well-known that the Hodge Laplacian on $(M,g)$ preserves the decomposition \eqref{eq:forms.hol.decomp}, so we obtain corresponding splittings of the spaces of eigenforms:
\begin{equation}\label{eq:eigenforms.hol}
    \mathcal{E}^k(\lambda)=\sum_j\mathcal{E}^k_{p_j}(\lambda).
\end{equation}

This is somewhat analogous to the situation in 4 dimensions, with the splitting of 2-forms into self-dual and anti-self-dual spaces, and corresponding decompositions of the eigenforms \eqref{eq:E2+-}.  
It is therefore reasonable to expect similar methods to those employed in the hyperk\"ahler 4-manifold setting to extend to other special holonomy manifolds.  In particular, we shall use the following on several occasions, which we can view as an extension of Lemma \ref{lem:hk.asd.exact}.

\begin{lem}\label{lem:exact.vanish}
    Let $(M^n,g)$ be an AC or A$T^m$C Riemannian $n$-manifold with a bounded $(n-4)$-form $\eta$ satisfying $\d\eta=0$.  Let $\alpha\in L^2\Omega^1(M)$ such that $\d\alpha\in L^2\Omega^2(M)$ satisfies
    \begin{equation}\label{eq:ASD.shol}
        \d\alpha\wedge \eta=\mu *\d\alpha
    \end{equation}
    for some $\mu\neq 0$.  Then $\d\alpha=0$.
\end{lem}

\begin{rmk}\label{rmk:vanish}
    We state Lemma \ref{lem:exact.vanish} in the setting of AC and A$T^m$C manifolds since these are our main interest, but as we shall see from the proof it clearly extends to other types of complete non-compact Riemannian manifolds (just like Lemma \ref{lem:hk.asd.exact}).  In the case of special holonomy, $\eta$ in Lemma \ref{lem:exact.vanish} will be a natural parallel form (e.g.~the K\"ahler form in the case of Calabi--Yau 3-folds).
\end{rmk} 

\begin{proof}
By \eqref{eq:ASD.shol} we have that
\begin{equation}\label{eq:mu.dalpha}
    \mu\|\d\alpha\|_{L^2}^2=\mu\int_M\d\alpha\wedge *\d\alpha=\int_M\d\alpha\wedge\d\alpha\wedge\eta=\int_M\d(\alpha\wedge\d\alpha\wedge\eta),
\end{equation}
since $\eta$ is closed.  As we are assuming that $\eta$ is bounded and $\alpha,\d\alpha\in L^2$, the same argument as in the proof of Lemma \ref{lem:hk.asd.exact} shows that the boundary term at infinity arising from applying Stokes' theorem as in \cite{Gaffney} to the final term in  \eqref{eq:mu.dalpha} vanishes.  Hence, 
\begin{equation}
    \mu\|\d\alpha\|_{L^2}^2=0
\end{equation}
which forces $\d\alpha=0$ as $\mu\neq 0$.
\end{proof}

\subsection{Calabi--Yau 3-folds}

We begin with Calabi--Yau 3-folds, which have holonomy in $\mathrm{SU}(3)$, and first recall the splitting of forms in this case.

\begin{lem}\label{lem:CY3.forms} Let $M$ be a Calabi--Yau $3$-fold with K\"ahler form $\omega$ and holomorphic volume form $\Omega$.  Then we have orthogonal decompositions:
\begin{align}
    \Omega^2(M)&=\Omega^2_1(M)\oplus\Omega^2_6(M)\oplus \Omega^2_8(M);\\
\Omega^3(M)&=\Omega^3_{1\oplus 1}(M)\oplus\Omega^3_6(M)\oplus\Omega^3_{12}(M),
\end{align}
where
\begin{align}
    \Omega^2_1(M)&=\{f\omega\,:\,f\in\Omega^0(M)\};\\
    \Omega^2_6(M)&=\{*(\alpha\wedge\mathrm{Re}\,\Omega)\,:\,\alpha\in\Omega^1(M)\}=\{\beta\in\Omega^2(M)\,:\,\beta\wedge\omega=*\beta\};\label{eq:Omega26}\\
    \Omega^2_8(M)&=\{\beta\in\Omega^2(M)\,:\,\beta\wedge\omega=-*\beta\};\label{eq:Omega28}\displaybreak[0]\\
    \Omega^3_{1\oplus 1}(M)&=\{f\mathrm{Re}\,\Omega\,:\,f\in\Omega^0(M)\}\oplus 
    \{f\mathrm{Im}\,\Omega\,:\,f\in\Omega^0(M)\};\\
    \Omega^3_6(M)&=\{\alpha\wedge\omega\,:\,\alpha\in\Omega^1(M)\};\\
    \Omega^3_{12}(M)&=\{\gamma\in\Omega^3(M)\,:\,\gamma\wedge\omega=0=\gamma\wedge\Omega\}.
\end{align}
\end{lem}

We can now give our vanishing results for Calabi--Yau 3-folds.

\begin{lem}\label{lem:CY3.vanish} Let $M$ be an AC or A$T^m$C Calabi--Yau $3$-fold and recall the notation of \eqref{eq:eigenforms.hol} and Lemma \ref{lem:CY3.forms}. Suppose that $\lambda\neq 0$ such that $\mathcal{E}^0(\lambda)=0$. 
\begin{itemize}
    \item[(a)] If $\mathcal{E}^2_8(\lambda)=0$ then $\mathcal{E}^k(\lambda)=0$ for all $k$.
    \item[(b)] If $\mathcal{E}^1(\lambda)=0=\mathcal{E}^3_{12}(\lambda)$ then $\mathcal{E}^k(\lambda)=0$ for all $k$.
\end{itemize}
\end{lem}

\begin{proof}
 Since $\mathcal{E}^0(\lambda)=0$ and $\Delta(f\omega)=(\Delta f)\omega$ for $f\in\Omega^0(M)$, we deduce that $\mathcal{E}^2_1(\lambda)=0$.  Similarly, $\mathcal{E}^3_{1\oplus 1}(\lambda)=0$. 

 We first show that if $\mathcal{E}^2_8(\lambda)=0$ then $\mathcal{E}^1(\lambda)=0$. 
 Suppose that $\alpha\in\mathcal{E}^1(\lambda)$. Then by Lemma \ref{lem:lambda.non.neg}, $\d^*\alpha\in\mathcal{E}^0(\lambda)=0$, so $\d^*\alpha=0$.  Moreover, $\d\alpha\in\mathcal{E}^2(\lambda)$.   If we are in case (a), i.e.~$\mathcal{E}^2_8(\lambda)=0$, then as $\mathcal{E}^2_1(\lambda)=0$ as well, we deduce that $\d\alpha\in\mathcal{E}^2_6(\lambda)$.  Using \eqref{eq:Omega26} and Lemma \ref{lem:exact.vanish} with $\eta=\omega$, we have that
  $\d\alpha=0$. Since $\d^*\alpha=0$ as well and $\lambda\neq 0$ we conclude that $\alpha=0$, so $\mathcal{E}^1(\lambda)=0$.  

 Suppose now that $\mathcal{E}^1(\lambda)=0$. 
 Then since $\omega$ and $\Omega$ are parallel we  deduce from Lemma \ref{lem:CY3.forms} that $\mathcal{E}^2_6(\lambda)=\mathcal{E}^3_6(\lambda)=0$.  Using the Hodge star, this only leaves us to consider $\beta\in\mathcal{E}^2_8(\lambda)$ and $\gamma\in \mathcal{E}^3_{12}(\lambda)$. 
 
 If $\mathcal{E}^3_{12}(\lambda)=0$ then we see that $\d\beta\in\mathcal{E}^3(\lambda)$, which we have shown is zero, so $\d\beta=0$.  Equation \eqref{eq:Omega28} then gives that
 \begin{equation}
     \d*\beta=-\d(\beta\wedge\omega)=0,
 \end{equation}
 and so we deduce that $\beta=0$ as $\lambda\neq 0$.  This proves (b).  
 
 If, on the other hand, $\mathcal{E}^2_8(\lambda)=0$ then $\d\gamma\in \mathcal{E}^4(\lambda)=0$ and $\d^*\gamma\in\mathcal{E}^2(\lambda)=0$.  This means $\gamma$ is harmonic which forces $\gamma=0$ as $\lambda\neq 0$. This completes the proof of (a).
\end{proof}

\begin{rmk}
    It is possible to show that if $\mathcal{E}^0(\lambda)=\mathcal{E}^1(\lambda)=0$ for $\lambda\neq 0$ on an  A$T^m$C Calabi--Yau 3-fold, then $\mathcal{E}^3_{12}(\lambda)\cong \mathcal{E}^2_8(\lambda)\oplus\mathcal{E}^2_8(\lambda)$,   by taking $(\beta_1,\beta_2)\in \mathcal{E}^2_8(\lambda)\oplus\mathcal{E}^2_8(\lambda)$ to $\d\beta_1+*\d\beta_2$.
\end{rmk}

\subsection{\texorpdfstring{G\textsubscript{2}}{G2} manifolds}

We now move on to $\text{G}_2$ holonomy and give the splitting of forms in this case.

 \begin{lem}\label{lem:G2.forms}
Let $(M^7,\varphi)$ be a $\text{G}_2$ manifold.  Then we have orthogonal decompositions:
\begin{align}
    \Omega^2(M)&=\Omega^2_7(M)\oplus \Omega^2_{14}(M);\\
\Omega^3(M)&=\Omega^3_{1}(M)\oplus\Omega^3_7(M)\oplus\Omega^3_{27}(M),
\end{align}
where
\begin{align}
    \Omega^2_7(M)&=\{*(\alpha\wedge*\varphi)\,:\,\alpha\in\Omega^1(M)\}=\{\beta\in\Omega^2(M)\,:\,\beta\wedge\varphi=2*\beta\};\label{eq:Omega27}\\
    \Omega^2_{14}(M)&=\{\beta\in\Omega^2(M)\,:\,\beta\wedge\varphi=-*\beta\};\label{eq:Omega214}\displaybreak[0]\\
    \Omega^3_{1}(M)&=\{f\varphi\,:\,f\in\Omega^0(M)\};\\
    \Omega^3_7(M)&=\{*(\alpha\wedge\varphi)\,:\,\alpha\in\Omega^1(M)\};\\
    \Omega^3_{27}(M)&=\{\gamma\in\Omega^3(M)\,:\,\gamma\wedge\varphi=0=\gamma\wedge*\varphi\}.
\end{align}
\end{lem}

We now have a vanishing result analogous to Lemma \ref{lem:CY3.vanish} in the $\text{G}_2$ setting.

\begin{lem}\label{lem:G2.vanish} Let $(M^7,\varphi)$ be an AC or A$T^m$C $\text{G}_2$ manifold,  and recall the notation of \eqref{eq:eigenforms.hol} and Lemma \ref{lem:G2.forms}. Suppose that $\lambda\neq 0$ such that $\mathcal{E}^0(\lambda)=0$. 
\begin{itemize}
    \item[(a)] If $\mathcal{E}^2_{14}(\lambda)=0$ then $\mathcal{E}^1(\lambda)=0$.
    \item[(b)] If $\mathcal{E}^1(\lambda)=0=\mathcal{E}^3_{27}(\lambda)$ then $\mathcal{E}^k(\lambda)=0$ for all $k$.
\end{itemize}
\end{lem}

\begin{proof}  The proof follows the same strategy as Lemma \ref{lem:CY3.vanish}.  The assumption $\mathcal{E}^0(\lambda)=0$ means that $\mathcal{E}^3_1(\lambda)=0$.  Note that $\mathcal{E}^1(\lambda)\cong\mathcal{E}^2_7(\lambda)\cong\mathcal{E}^3_7(\lambda)$ by Lemma \ref{lem:G2.forms} as $\varphi$ and $*\varphi$ are parallel.

Let $\alpha\in\mathcal{E}^1(\lambda)$.  Then $\d^*\alpha\in\mathcal{E}^0(\lambda)$ so $\d^*\alpha=0$.  If we assume that $\mathcal{E}^2_{14}(\lambda)=0$ then $\d\alpha\in\mathcal{E}^2_7(\lambda)$.  Using \eqref{eq:Omega27} and applying Lemma \ref{lem:exact.vanish} with $\eta=\varphi$ implies that $\d\alpha=0$.  
 Hence $\alpha=0$ as $\lambda\neq 0$, which proves (a).

If $\mathcal{E}^1(\lambda)=0$ then, by our earlier observation,  
$\mathcal{E}^2_7(\lambda)=\mathcal{E}^3_7(\lambda)=0$.  If we also assume that $\mathcal{E}^3_{27}(\lambda)=0$ then we are only concerned with $\beta\in\mathcal{E}^2_{14}(\lambda)$.  However, $\d\beta\in\mathcal{E}^3(\lambda)=0$ and \eqref{eq:Omega214} yields
\begin{equation}
\d*\beta=-\d(\beta\wedge\varphi)=0,
\end{equation}
so again we deduce that $\beta=0$ as $\lambda\neq 0$.  This completes the proof of (b).
\end{proof}

\subsection{Spin(7) manifolds}

We conclude our discussion by looking at 8-manifolds with holonomy contained in Spin(7), which is slightly different to the previous cases.

\begin{lem}\label{lem:Spin7.forms}
Let $(M^8,\Phi)$ be a $\text{\emph{Spin}}(7)$ manifold.  Then we have orthogonal decompositions:
\begin{align}
    \Omega^2(M)&=\Omega^2_7(M)\oplus \Omega^2_{21}(M);\\
\Omega^3(M)&=\Omega^3_8(M)\oplus\Omega^3_{48}(M);\\
\Omega^4(M)&=\Omega^4_1(M)\oplus\Omega^4_7(M)\oplus\Omega^4_{27}(M)\oplus\Omega^4_{35}(M);
\end{align}
where
\begin{align}
    \Omega^2_7(M)&=\{\beta\in\Omega^2(M)\,:\,\beta\wedge\Phi=3*\beta\};\label{eq:Omega27.spin}\\
    \Omega^2_{21}(M)&=\{\beta\in\Omega^2(M)\,:\,\beta\wedge\Phi=-*\beta\}\label{eq:Omega221};\displaybreak[0]\\
    \Omega^3_8(M)&=\{*(\alpha\wedge\Phi)\,:\,\alpha\in\Omega^1(M)\};\displaybreak[0]\\
      \Omega^4_{1}(M)&=\{f\Phi\,:\,f\in\Omega^0(M)\};\\
      \Omega^4_7(M)&\cong\Omega^2_7(M);\\
      \Omega^4_{35}(M)&=\{\eta\in\Omega^4(M)\,:\eta=-*\eta\}.
\end{align}
Note  that $\Omega^4_{27}(M)$ is the orthogonal complement of $\Omega^4_1(M)\oplus\Omega^4_7(M)$ in the space of self-dual $4$-forms on $M$:
\begin{equation}
    \Omega^4_+(M)=\Omega^4_1(M)\oplus\Omega^4_7(M)\oplus\Omega^4_{27}(M).
\end{equation}
\end{lem}

We can now state our vanishing result in the Spin(7) case.

\begin{lem}\label{lem:Spin7.vanish} Let $(M^8,\Phi)$ be an AC or  A$T^m$C $\text{\emph{Spin}}(7)$-manifold  and recall the notation of \eqref{eq:eigenforms.hol} and Lemma \ref{lem:Spin7.forms}. Suppose that $\lambda\neq 0$ such that $\mathcal{E}^0(\lambda)=0$. 
\begin{itemize}
    \item[(a)] If $\mathcal{E}^2_7(\lambda)=0$ or $\mathcal{E}^2_{21}(\lambda)=0$ then $\mathcal{E}^1(\lambda)=0$.
    \item[(b)] If $\mathcal{E}^1(\lambda)=0=\mathcal{E}^3_{48}(\lambda)$   then $\mathcal{E}^k(\lambda)=0$ for all $k$.
\end{itemize}
\end{lem}

\begin{proof}
We again follow the same proof strategy as Lemmas \ref{lem:CY3.vanish} and \ref{lem:G2.vanish} with slight differences.  Notice first that if $\mathcal{E}^0(\lambda)=0$, then $\mathcal{E}^4_1(\lambda)=0$ since $\Phi$ is parallel.

Suppose that $\alpha\in\mathcal{E}^1(\lambda)$.  As before, we have that $\d^*\alpha=0$ using $\mathcal{E}^0(\lambda)=0$.  If $\mathcal{E}^2_7(\lambda)=0$ or $\mathcal{E}^2_{21}(\lambda)=0$ then \eqref{eq:Omega27.spin}--\eqref{eq:Omega221} imply that we can apply Lemma \ref{lem:exact.vanish} with $\eta=\Phi$ to deduce that $\d\alpha=0$.  
Since $\alpha$ is harmonic and $\lambda\neq 0$ this implies $\alpha=0$, which gives (a).

For (b), if $\mathcal{E}^1(\lambda)=0$ then $\mathcal{E}^3_8(\lambda)=0$ by Lemma \ref{lem:Spin7.forms}.  Suppose $\mathcal{E}^3_{48}(\lambda)=0$.  Then $\mathcal{E}^3(\lambda)=0$.  Therefore, if $\beta\in\mathcal{E}^2(\lambda)$ then $\d\beta\in\mathcal{E}^3(\lambda)=0$, so $\d\beta=0$.  Similarly, $\d^*\beta\in\mathcal{E}^1(\lambda)=0$ so $\beta=0$ as $\lambda\neq 0$.  If $\eta\in\mathcal{E}^4(\lambda)$ then $\d^*\eta\in\mathcal{E}^3(\lambda)=0$ and $\d\eta\in\mathcal{E}^5(\lambda)$ which vanishes since $\mathcal{E}^3(\lambda)=0$.  We deduce that $\mathcal{E}^4(\lambda)=0$ which gives that $\mathcal{E}^k(\lambda)=0$ for all $k$ in this case.
\end{proof}

\begin{rmk}   By Remark \ref{rmk:vanish}, the vanishing results given in Lemmas \ref{lem:CY3.vanish}, \ref{lem:G2.vanish} and \ref{lem:Spin7.vanish}  extend to other asymptotic behaviours beyond AC and A$T^m$C.     
    It would be interesting to see if one can find examples of complete non-compact  manifolds with holonomy $\mathrm{SU}(3)$, $\mathrm{G}_2$ or $\mathrm{Spin}(7)$ which have no $L^2$ eigenfunctions or 1-forms, but still do have non-trivial $L^2$ eigenforms.  Proposition \ref{prop:AC.Ek} shows this cannot occur in the AC case.
\end{rmk}

\section{Physical interpretation}\label{sec:physics}

In this final section we give some brief physical interpretations of our results in the case of gravitational instantons and discuss physical predictions for some Calabi--Yau 3-folds and $\mathrm{G}_2$ manifolds. 

\subsection{String/\texorpdfstring{$\mathbf{M}$}{M}-theory on ALE gravitational instantons}  Recall that, as noted in Remark \ref{rmk:ALE}, any ALE gravitational instanton may be denoted $(M^4_{\Gamma_{ADE}},g)$ as it is asymptotic to $\R^4/\Gamma_{ADE}$ for an appropriate finite subgroup $\Gamma_{ADE}\subseteq\mathrm{SU}(2)$.

In physics terms, Corollary \ref{cor:ALE} shows that ALE hyperk\"ahler 4-manifolds  
 have no massive  $L^2$-normalisable modes. This means that their $L^2$ spectrum leads only to harmonic forms, which are known \cite{Hausel} to be given by anti-self-dual 
 2-forms spanning the compactly supported cohomology 
 \begin{equation}
 H^2_{\text{cs}}(M^4_{\Gamma_{ADE}},\mathbb{R)} \cong \mathbb{R}^{rk(ADE)}.
 \end{equation}
 This cohomology can be naturally viewed as the Cartan subalgebra $\mathfrak{h}_{ADE}(\mathbb{R}) \subset \mathfrak{g}_{ADE}(\mathbb{R})$ for the $ADE$ Lie algebra. In fact, the moduli space ${\cal{M}}_{ADE}$ of ALE hyperk\"ahler structures on $M^4_{\Gamma_{ADE}}$, including orbifolds, is given by three copies of $\mathfrak{h}_{ADE}(\mathbb{R})$, up to the action of the Weyl group:
\begin{equation}\label{eq:M.ADE}
    {\cal{M}}_{ADE} = \mathfrak{h}_{ADE}(\mathbb{R})^3/W_{ADE}.
\end{equation}
In particular, if we fix any hyperk\"ahler structure $\{\omega_1,\omega_2,\omega_3\}$ in ${\cal{M}}_{ADE}$, then any neighbouring hyperk\"ahler structure is given by deforming each of the $\omega_j$ by any of the $L^2$  harmonic 2-forms. Thus all of the moduli (i.e.~deformation parameters) are $L^2$-normalisable. 

 ALE gravitational instantons $(M^4_{\Gamma_{ADE}},g)$ have numerous applications and interpretations in superstring/M-theory where they appear as exact vacuum spacetime solutions of the theory. In particular, M-theory describes physics in 11-dimensional Lorentzian spacetimes 
 $(\hat{M}^{10,1}, \hat{g}^{10,1})$ and we have exact spacetime vacua of the form 
 \begin{equation}\label{eq:ALE.vacua}
 (\hat{M}^{10,1}, \hat{g}^{10,1})=(M^4_{\Gamma_{ADE}} \times \mathbb{R}^{6,1},  g^4 + \eta^{6,1}),
 \end{equation}
 with $g^4=g$ and $\eta^{6,1}$ the Minkowski metric. 
 
 It is well known that, at the level of zero-modes, the physics of M-theory on this spacetime is governed by M-theory in the bulk coupled to 7d super Yang--Mills theory with $ADE$ gauge group \cite{Hull, Acharya}. This is borne out by the fact that the moduli space of the latter theories is precisely \eqref{eq:M.ADE} and also requires adding non-perturbative states in the form of M2-branes wrapping the calibrated homology cycles of $(M^4_{\Gamma_{ADE}}, g)$.  Corollary \ref{cor:ALE} establishes rigorously that there are no massive Kaluza--Klein particles as they would arise from massive $L^2$ eigenforms.



In other words, the physics of the M-theory spacetime vacuum \eqref{eq:ALE.vacua} 
consists of 11-dimensional supergravity in the bulk coupled to a 7d supersymmetric $ADE$ gauge theory. There are no massive Kaluza--Klein states.

At the end of the day we prove that M-theory on $(M^4_{\Gamma_{ADE}}, g)$ geometrically engineers 7d super Yang--Mills theory with $ADE$ Lie algebra and nothing else.
 Similarly, if we instead consider Type IIB theory on $(M^4_{\Gamma_{ADE}}, g)$ we prove that the physics is exactly given by bulk Type IIB string theory coupled to the $ADE$ $(2,0)$-theories and nothing else.

\subsection{ALF gravitational instantons}

We can also consider physical questions for the other known gravitational instantons. 
Here we will restrict our attention to the ALF case, leaving the other cases for future investigations.

Physics makes the following predictions.
\begin{itemize}
\item For multi-Taub--NUT there are no $L^2$ eigenmodes of the scalar Laplacian since D0-branes cannot bind to D6-branes, but rather scatter off of them. We demonstrated that this is the case in Theorem \ref{thm:An.ALF}. 
\item For Atiyah--Hitchin and its double cover,  
D0-branes can now bind to O6-planes in the string theory picture, allowing for \emph{bound states} (which are interpreted as $L^2$ eigenfunctions/forms).  These expectations are borne out by Corollary \ref{cor:ALF}(b).
\end{itemize}

\subsection{Higher dimensions: Calabi--Yau 3-folds and \texorpdfstring{G\textsubscript{2}}{G2} manifolds}
 
Together, Theorem \ref{thm:grav.inst} and Corollaries \ref{cor:ALE} and \ref{cor:ALF} give us a good understanding of $L^2$ eigenforms on gravitational instantons, particularly in the ALE and ALF settings, but our results for complete non-compact Ricci-flat manifolds with special holonomy in higher dimensions are weaker.  
It is therefore interesting to ask: what can physics predict in these higher-dimensional settings?  

In higher dimensions, ALE naturally generalises to asymptotically conical (AC), and here we know that we have no massive (i.e.~non-zero eigenvalue) $L^2$ eigenforms  by Proposition \ref{prop:AC.Ek}, whether we have special holonomy or not.  However, physics predicts the existence of massive $L^2$ eigenforms as follows.

Suppose $(M, g(a))$ is an AC Calabi--Yau 3-fold or $\text{G}_2$ manifold  with a moduli space of Ricci-flat metrics labelled by parameters collectively denoted as $a$. Suppose further that along some subset in the moduli space (corresponding to certain values of $a$), $(M, g(a))$ develops codimension four  orbifold singularities which are localised on a compact codimension four submanifold, $S \subset M$.  Examples of such families of AC Calabi--Yau 3-folds are discussed in \cite{Acharya2}.  

In this situation, $S$ is the locus of an $ADE$ singularity.  For large enough $S$ we should have a copy of 7d super Yang--Mills theory on $S$ times a non-compact Minkowski spacetime.  Moreover, since $S$ is compact, there should be massive Kaluza--Klein particles.  However, Proposition \ref{prop:AC.Ek} precludes the existence of such particles, so some discussion of the physical interpretation is required. First note that there is no contradiction because $(M, g(a))$ is singular on this subspace of the moduli space. Furthermore, the physical interpretation as a Yang--Mills theory requires the inclusion of branes, which are non-perturbative states, not evident from our analysis of the Laplacian eigenforms. 

As a proof of principle, in the examples of \cite{Acharya2}, the singular Calabi--Yau 3-folds in question are discrete quotients by $\Gamma_{ADE}$ of the resolved conifold with the AC metric constructed in \cite{CandelasDeLaOssa}. Since Proposition \ref{prop:AC.Ek} asserts that this AC Calabi--Yau admits no massive $L^2$ eigenforms, neither does the quotient. However, there are massless branes at the singularity, so presumably the would-be modes in question arise from these non-perturbative states.

Similarly one can imagine AC $\text{G}_2$ manifolds which develop codimension six singularities along  circles as one varies the parameters $a$. In this case, the codimension six singularity is believed to support a five-dimensional conformal field theory, hence such theories compactified on a circle should also lead to massive Kaluza--Klein modes.  Again, this is a singular locus in the moduli space and branes are required for the physical interpretation, thereby avoiding any contradiction with
Proposition \ref{prop:AC.Ek}.

\bibliographystyle{plain}
\bibliography{L2eigenforms}

\bigskip

\begingroup
\small

\noindent
\textsc{Bobby S. Acharya}\\
The Abdus Salam International Centre for Theoretical Physics (ICTP),\\
Strada Costiera 11, I-34151 Trieste, Italy\\
\href{mailto:bacharya@ictp.it}{\texttt{bacharya@ictp.it}}

\medskip

\noindent
\textsc{Luc\'ia M. Cabrera}\\
Instituto Balseiro, Universidad Nacional de Cuyo and
Comisi\'on Nacional de Energ\'ia At\'omica,
Av.\ Bustillo km 9.5, R8402AGP San Carlos de Bariloche,
R\'io Negro, Argentina.\\
\href{mailto:lucia.cabrera@ib.edu.ar}
{\texttt{lucia.cabrera@ib.edu.ar}}

\medskip

\noindent
\textsc{Simone Corbo}\\
Scuola Internazionale Superiore di Studi Avanzati (SISSA),
Via Bonomea 265, 34136 Trieste, Italy and INFN, Sezione di Trieste.\\
\href{mailto:scorbo@sissa.it}{\texttt{scorbo@sissa.it}}

\medskip

\noindent
\textsc{Jason D. Lotay}\\
Mathematical Institute, University of Oxford,
Andrew Wiles Building, Radcliffe Observatory Quarter,
Woodstock Road, Oxford OX2 6GG, United Kingdom.\\
\href{mailto:jason.lotay@maths.ox.ac.uk}{\texttt{jason.lotay@maths.ox.ac.uk}}
\endgroup

\end{document}